\documentclass[journal,twoside,web]{ieeecolor}

\usepackage{generic}
\usepackage{cite}
\usepackage{amsmath,amssymb,amsfonts}
\usepackage{algorithmic}
\usepackage{graphicx}
\usepackage{booktabs}
\usepackage{caption}
\usepackage{makecell}
\usepackage{array}
\usepackage{algorithm,algorithmic}
\usepackage{epstopdf}
\usepackage{hyperref}
\hypersetup{hidelinks=true}
\usepackage{mathrsfs}
 \usepackage{dsfont}
\DeclareMathOperator{\col}{col}
\DeclareMathOperator{\sign}{sign}
\DeclareMathOperator{\sgn}{sgn}
\DeclareMathOperator{\sig}{sig}
\DeclareMathOperator{\diag}{diag}
\usepackage{textcomp}
\newtheorem{thm}{Theorem}[section]
\newtheorem{corollary}{Corollary}[section]
\newtheorem{lem}{Lemma}[section]

\newtheorem{defn}{Definition}[section]

\newtheorem{assum}{Assumption}[section]
\newtheorem{rem}{Remark}[section]

\def\BibTeX{{\rm B\kern-.05em{\sc i\kern-.025em b}\kern-.08em
    T\kern-.1667em\lower.7ex\hbox{E}\kern-.125emX}}
\begin{document}
\title{Robust Nash Equilibrium Seeking Based on Semi-Markov Switching Topologies}
\author{Jianing Chen, Sitian Qin and Chuangyin Dang$^{*}$, \IEEEmembership{Senior Member, IEEE}
\thanks{This work was supported in part by the National Natural Science Foundation of China (62176073,12271127), and the Natural Science Foundation of Shandong Province (ZR2024MF080). Corresponding author: Chuangyin Dang.}
\thanks{J. Chen and S. Qin are with the Department of Mathematics, Harbin Institute of Technology, Weihai, 264209, China (e-mail: chjn6511@163.com; qinsitian@163.com).}
\thanks{C. Dang is with the Department of Systems Engineering, City University of Hong Kong, Hong Kong (e-mail: mecdang@cityu.edu.hk).}
}
\maketitle

\begin{abstract}
This paper investigates a distributed robust Nash Equilibrium (NE) seeking problem for second-order players subject to external disturbances and uncertain dynamics while communicating via semi-Markov switching topologies. To accommodate the above concerns, the following targets require to be reached simultaneously: (1) Disturbances and uncertain dynamics rejection in finite time; (2) NE seeking for the second-order players; (3) Distributed action estimation on non-neighboring players under semi-Markov switching. By combining supertwisting-based Integral Sliding-Mode Control (ISMC) with a leader-follower consensus protocol, a novel robust NE seeking algorithm is constructed. Furthermore, to lessen dispensable information transmission, a sampled-data-based event-triggered mechanism is introduced. Incorporating the advantages of both semi-Markov switching and event-triggered mechanism, another NE seeking algorithm is proposed. Theoretical analysis via a Lyapunov-Krasovskii functional proves the leader-follower consensus can be achieved in the mean-square sense. Finally, a connectivity control game is formulated to validate the algorithms.
\end{abstract}


\section{Introduction}
\IEEEPARstart{R}{ecently}, distributed NE seeking for the noncooperative game has gained much attention owing to its expansive application areas (see \cite{deng2021distributed,ye2017distributed,krilavsevic2021learning,salehisadaghiani2016distributed,li2024logical,yan2024incorporation}), where multiple independent but selfish players compete with each other to minimize their own cost functions by local information exchange. Up until now, many excellent research findings involving distributed NE seeking have been put forth \cite{deng2021distributed,ye2017distributed,krilavsevic2021learning,salehisadaghiani2016distributed,li2024logical,yan2024incorporation,wang2022robust,yuan2022event,hu2022distributed,bianchi2021continuous}.

Although the existing distributed NE seeking algorithms are capable of handling many problems, it is still essential to consider the physical dynamics of players during algorithms design. In modern engineering, noncooperative games are increasingly deployed on physical entities, such as YouBot vehicles in connectivity control games \cite{li2022distributed} or turbine generators in energy transmission problems \cite{cai2022nash}. In these practical applications, a player's strategy intrinsically corresponds to its physical state, such as spatial position or velocity. Consequently, the strategy updating process is strictly governed by the laws of Newtonian mechanics \cite{kamal2021control}, rendering abstract first-order systems used in \cite{wang2022robust,yuan2022event,hu2022distributed} inadequate. As a result, recent research has increasingly focused on NE-seeking strategies integrated with cyber-physical systems, addressing more realistic and intricate player dynamics (e.g., inertia or acceleration), such as second-order systems in \cite{ibrahim2018nash,ye2020distributed,liu2022distributed,ye2021distributed,deng2022generalized}. For example, the NE seeking problem with nonlinear second-order player dynamics was first investigated in \cite{ibrahim2018nash}, which effectively modeled the heavy-ball friction dynamics. Subsequently, based on velocity estimators, velocity measurement-free distributed NE seeking algorithms were proposed in \cite{ye2021distributed}, which fulfilled the robot tracking problem without additional velocity measurement devices. However, unlike ideal first-order algorithmic updates that only operate in the cyber layer, the physical execution of strategies in second-order systems is inevitably subjected to complex real-world environments. When physical players (e.g., YouBot vehicles) adjust their positions to seek the NE, they frequently encounter external environmental disturbances (e.g., wind gusts, unknown road friction) and unmodeled mechanical dynamics \cite{li2022distributed}. Acting directly on the input channels, these physical uncertainties will cause severe steady-state deviations and prevent players from converging to the true NE.

Fortunately, strides have been made to address external disturbances and uncertain dynamics during distributed NE seeking \cite{li2022distributed,ye2020distributed,wang2020distributed,huang2020distributed,zhang2019distributed}. For instance, robust NE seeking algorithms were developed in \cite{ye2020distributed} by treating unknown disturbances and the nonlinear dynamics of first-order systems as extended states. Similarly, the authors in \cite{wang2020distributed} considered second-order players subjected to time-varying disturbances, unknown dynamics, and attacks, proposing an NE seeking algorithm inspired by hybrid systems. However, the RISE-based observers in \cite{ye2020distributed,wang2020distributed} strictly require the second-order time derivatives of the disturbances to be bounded. To alleviate such restrictions on player dynamics, disturbances, and communication topologies, high-gain observers were introduced in \cite{huang2020distributed}, where the nonlinear dynamics and time-varying disturbances are required to be globally Lipschitz and bounded, respectively. Although this approach relaxes the requirement for bounded second-order derivatives, it is strictly confined to quadratic games and incurs high communication costs, thereby restricting its applicability in large-scale distributed systems. Alternatively, low-pass filter-based methods were developed in \cite{li2022distributed} to compensate for disturbances and unknown dynamics with bounded first derivatives. However, despite their methodological differences, these aforementioned methods typically guarantee only asymptotic disturbance rejection, meaning the residual influence of these disturbances persists throughout the transient phase. In practical systems like power grids \cite{knudsen2015dynamic}, the lack of rapid, finite-time robust control can trigger system collapse during frequent condition shifts. Motivated by these limitations, this paper aims to establish novel compensating mechanisms to relax existing assumptions and enhance the rate of disturbance rejection.

So far, the interferences in NE seeking caused by the physical features of multi-agent systems have been well discussed. However, most of these discussions assume a stable communication network with a fixed topology. In reality, topologies change constantly due to link failures \cite{lobel2010distributed}, cyber-attacks \cite{zhao2019resilient}, and environmental shifts \cite{guo2023stabilization}. Empirical observations reveal that link disconnections caused by hardware aging, battery depletion, and DoS attacks inherently exhibit specific statistical distribution characteristics \cite{akyildiz2002wireless, moore2006inferring}. Conventionally, researchers have relied on Markov chains to model these random topological variations \cite{fang2019distributed, zhao2012distributed}. Nevertheless, subsequent studies have uncovered that the standard Markov process-based formulation presents a fundamental limitation that hinders its applicability in real-world scenarios: it strictly relies on the memoryless property, assuming that the sojourn time of any topological state follows an exponential distribution. This may contradict engineering realities. In practical communication networks, the occurrences and recoveries of communication links are heavily time-dependent. For instance, empirical reliability studies demonstrate that sensor node failure times due to hardware aging or battery depletion typically follow Weibull or log-normal distributions rather than memoryless exponential ones \cite{liu2021low}. Similarly, in the context of cyber-security, the durations of cyber-attacks (e.g., jamming or DoS) are strictly constrained by the adversaries' finite energy and resources. Consequently, the probability of the network recovering from a jammed state becomes heavily dependent on the sojourn time of the attack, often exhibiting a heavy-tailed distribution \cite{moore2006inferring}. Therefore, to overcome this memoryless limitation and accurately represent practical communication networks, this paper adopts semi-Markov switching topologies, allowing the sojourn times to follow more general and realistic distributions (e.g., Weibull).

Beyond network instability, practical limitations on communication resources also arise. To significantly alleviate communication burden, event-triggered mechanisms have emerged as a powerful paradigm, enabling players to broadcast their information only when predefined triggering conditions are activated \cite{cai2023distributed}. Consequently, various distributed event-triggered NE seeking algorithms have been explored utilizing either static event-triggered mechanisms \cite{yuan2022event,shi2019distributed,yuan2017event} or dynamic ones \cite{zhang2021distributed,zhang2021nash,zhangandPengandYuanYuanandLiuHuapingandGaoZhan2022,cai2023distributed,xu2022hybrid,9881225}. However, the direct application of these mechanisms reveals fundamental limitations. First, both static and dynamic mechanism frequently employ asymptotically decaying thresholds, which inevitably leads to an increasing communication frequency as the system converges. Second, evaluating these triggering conditions inherently requires players to constantly monitor their own states and internal variables, which consumes substantial sensing and computational resources. This practical gap motivates the incorporation of a sampled-data-based event-triggered mechanism in this paper. By restricting triggering evaluations solely to discrete sampling instants, it eliminates continuous monitoring and rigorously excludes Zeno behavior. Nevertheless, integrating this mechanism with random semi-Markov switching introduces substantial theoretical complexities; the discrete nature of the updates, combined with random network variations, can cause the players' states to exhibit severe deviations from their ideal continuous trajectories, thereby posing substantial theoretical challenges in the design and analysis of this event-triggered mechanism.

The primary contributions can be outlined as follows
\begin{itemize}
	\item A supertwisting-based ISMC scheme combined with average consensus tracking is proposed for finite-time rejection of disturbances and uncertain dynamics during NE seeking. Compared to RISE-based observers \cite{ye2020distributed,wang2020distributed} requiring bounded second-order disturbance derivatives, and high-gain observers \cite{huang2020distributed} confined to quadratic games, the proposed ISMC imposes only bounded first-derivative conditions, accommodating broader practical non-vanishing disturbances. Furthermore, unlike existing asymptotic compensators \cite{li2022distributed,ye2020distributed,wang2020distributed,huang2020distributed,zhang2019distributed}, the proposed finite-time compensating mechanism completely eliminates residual disturbance influences.
	\item As far as we know, this paper presents the first research on distributed NE seeking incorporating semi-Markov switching topologies, effectively addressing real-world challenges like link failures \cite{liu2021low} and cyber-attacks \cite{moore2006inferring}. Unlike standard Markov switching topologies \cite{fang2019distributed} restricted by memoryless properties with constant transition rates, the semi-Markov switching topologies accommodates general sojourn time distributions (e.g., Weibull). This relaxation introduces time-varying transition rates, which invalidate the static algebraic Lyapunov tools utilized in \cite{fang2019distributed}. To overcome this theoretical challenge, a novel piecewise stochastic Lyapunov-Krasovskii functional is employed to establish mean-square convergence of the algorithm. Additionally, by utilizing a continuous-time NE seeking algorithm, it avoids the complicated step-size tuning inherent in discrete-time algorithms \cite{fang2019distributed}.
	\item To the best of our knowledge, this paper is the first to design a sampled-data-based event-triggered mechanism for distributed robust NE seeking algorithm to reduce communication burdens. By imposing a discrete sampling period, the proposed mechanism strictly lower-bounds inter-event times, effectively avoiding the exponential decay of triggering intervals inherent to static event-triggered mechanisms \cite{yuan2022event,shi2019distributed,yuan2017event}. Compared to dynamic event-triggered mechanisms \cite{zhang2021distributed,zhang2021nash,zhangandPengandYuanYuanandLiuHuapingandGaoZhan2022,cai2023distributed,xu2022hybrid,9881225}, the proposed mechanism avoids the calculation of additional internal dynamic variables, thereby reducing the computational complexity while still effectively prolonging the triggering intervals. Crucially, almost all of the existing distributed event-triggered NE seeking algorithms \cite{yuan2022event,shi2019distributed,yuan2017event,zhang2021distributed,zhang2021nash,zhangandPengandYuanYuanandLiuHuapingandGaoZhan2022,cai2023distributed,xu2022hybrid,9881225} necessitate the constant trigger monitoring, which heavily exhausts the measurement and computational resources of the players, the proposed mechanism only evaluates the triggering condition at discrete sampling instants, making the algorithm highly amenable to practical implementation.
\end{itemize}
\par
The remainder of this paper is organized as follows. Section \ref{sec2} provides essential preliminaries and outlines the NE seeking problem in the presence of external disturbances and uncertain dynamics. In Sections \ref{sec3} and \ref{sec4}, we introduce two innovative robust distributed NE seeking algorithms: one without and one with semi-Markov switching topologies, and present their respective stability results. In Section \ref{sec5}, a connectivity control game is presented to validate the efficacy of the proposed algorithms. Finally, Section \ref{sec6} offers the conclusion drawn from the research.
\section{Preliminaries}\label{sec2}
In this section, essential foundational elements are presented, offering key background information and fundamental concepts.
\subsection{Notations}
$\mathbb{R}$, $\mathbb{R}^N$, and $\mathbb{R}^{N\times M}$ denote the set of real numbers, $N$-dimensional vectors and $N\times M$ matrices, respectively. $\mathbb{Z}_+$ denotes for all the non-negative integers. The inverse, the transpose and the Kronecker product are denoted as $(\cdot)^{-1}$, $(\cdot)^{\mathrm{T}}$ and $\otimes$, respectively. The column vector is denoted as col$\left(x_{1},x_{2},\ldots,x_{N}\right)=[x_1^{\rm T},x_2^{\rm T},\ldots,x_N^{\rm T}]^{\rm T}$. diag$\{d_{1},d_{2},\ldots,d_{N}\}$ denotes a diagonal matrix whose principal diagonal elements are $d_{i}$ for $i\in\{1,\ldots,N\}$. $\mathbf{0}_{N}:=\text{col}(0, 0,\ldots, 0)\in\mathbb{R}^{N}$, $\mathbf{1}_{N}:=\text{col}(1,1,\ldots,1)\in \mathbb{R}^{N}$. $\mathbf{I}_{N}\in\mathbb{R}^{N\times N}$ represents an identity matrix with $N$ dimension. $|\cdot|$, $\Vert \cdot \Vert$ and $\Vert\cdot\Vert_{\infty}$ stand for the absolute value, the Euclidean norm and the infinite norm, respectively. The component-wise signum function of a vector $x\in\mathbb{R}^n$ is denoted by $\sgn(x)=(\sign(x_1),\sign(x_2),\ldots,\sign(x_n))^{\mathrm{T}}$, where $\sign(x_i)=-1$ if $x_i<0$; $\sign(x_i)=1$ if $x_i>0$; $\sign(x_i)=0$, otherwise. Let $\sig^q(x)=(\sig^q(x_1),\sig^q(x_2),\ldots,\sig^q(x_n))^{\mathrm{T}}$, as well as $\sig^q(x_i)=\sign(x_i)|x_i|^q$, $\forall q\in(0,+\infty)$. For matrices $A, B\in\mathbb{R}^{n\times n}$, $A\preceq B$ ($A\succ B$) represents that $A-B$ is negative semi-definite (positive definite). $\mathbb{E}$ stands for the mathematical expectation operator.
\subsection{Graph Theory}
This subsection incorporates essential concepts from graph theory \cite[Definition 1]{ANZAI199213}. A weighted undirected graph, denoted as $\mathscr{G:=(V,E)}$, consists of a node set $\mathscr{V:=}\{1,2,\ldots$ $,\mathit{N}$\} and an edge set $\mathscr{E}\subseteq\{\{i,j\}|i,j\in\mathscr{V}\}$. For each node $i$, $\mathscr{N}_{i}:=\{\mathit{j|\{j,i\}}\in$ $\mathscr{E}$\} represents its neighbor set, where the set $\{j,i\}=\{i,j\}\in\mathscr{E}$ implies that node $\mathit{i}$ and node $\mathit{j}$ can communicate with each other. A path in the graph is a sequence of distinct nodes, with an edge connecting any consecutive pair of nodes. An undirected graph $\mathscr{G}$ is connected if there exists a path between any two nodes.
	\par
	The weighted adjacency matrix of the undirected graph $\mathscr{G}$ is represented by $\mathscr{A:=}$ [$\mathit{a_{ij}}$]$\in\mathbb{R}^{\mathit{N\times N}}$, where $a_{ii}=0,~a_{ij}=a_{ji}>0$ if $\mathit{j}\in\mathscr{N}_{i}$, and conversely, $a_{ij}=a_{ji}=0$. The degree matrix $\mathscr{D}$ of graph $\mathscr{G}$ is denoted as
	$\mathscr{D}=$diag$\{\sum_{\mathit{j}=1}^{\mathit{N}}\mathit{a_{1j}},\ldots,\sum_{\mathit{j}=1}^{\mathit{N}}\mathit{a_{Nj}}\}$. The Laplacian matrix $L$ is defined as $L:=\mathscr{D}-\mathscr{A}$. For an undirected graph, it is obvious that $\mathbf{1}_{\mathit{N}}^{\rm T}L=\mathbf{0}_{\mathit{N}}^{\rm T}$ and $L\mathbf{1}_{\mathit{N}}=\mathbf{0}_{\mathit{N}}$. Denote $\lambda_{1}(L), \lambda_{2}(L), \ldots, \lambda_{N}(L)$ as the eigenvalues of $L$. For an undirected connected graph, we can conclude that $0=\lambda_{1}(L)<\lambda_{2}(L)\leq \ldots \leq \lambda_{\mathit{N}}(L)$.
	\subsection{Convex Analysis and Dynamical System}
	Several fundamental definitions and lemmas in convex analysis and dynamical system are presented in this subsection.
	\begin{itemize}
		\item If
	$
	f(k x_{1}+(1-k)x_{2})\leq k f(x_{1})+(1-k) f(x_{2}),$~for any $x_{1}, x_{2} \in \mathbb{R}^{N},~k \in[0,1],
	$ then $f:\mathbb{R}^{N}\rightarrow\mathbb{R}$ is a convex function.
\item If there exists a $\theta>0$ such that $\left\| f(x_{1})-f(x_{2})\right\|\leq\theta\Vert x_{1}-x_{2} \Vert$, for any $x_{1}, x_{2}\in\mathbb{R}^{N} $, then $f:\mathbb{R}^{N}\to\mathbb{R}$ is a $\theta$-Lipschitz function.
	\item If there exists a $\mu>0$ such that $(x_{1}-x_{2})^{\mathrm{T}} (g(x_{1})-g(x_{2}))\geq \mu \Vert x_{1}-x_{2}\Vert ^{2}$, for any $x_{1}, x_{2}\in\mathbb{R}^{\mathit{N}}$, then $g:\mathbb{R}^{N} \rightarrow \mathbb{R}^{N}$ is a $\mu$-strongly monotone function.
	\end{itemize}
\begin{lem}\cite[Lemma 1]{yu2015finite}\label{yinli35}
Consider the dynamical system described as follows
\begin{equation}\label{xitong3}
\dot{x}(t)=f(t,x(t)),~f(t,0)=\mathbf{0}_m,~ x\in\mathcal{D}\subset\mathbb{R}^m.
\end{equation}
If there is a continuous differentiable and positive definite function $H(x)$ defined in a neighborhood of the origin $\mathcal{U}\subset\mathcal{D}$, and the solution of \eqref{xitong3} satisfies
\begin{equation*}
\dot{H}(x(t))\leq-\zeta H^a(x(t)),
\end{equation*}
where constant $a\in(0,1)$ represents the power and $\zeta>0$. Then, the origin of the system \eqref{xitong3} is finite-time stable. The settling time $T_{\max}$ is estimated by
\begin{equation*}
T\leq T_{\max}\doteq\frac{H^{1-a}(x(0))}{\zeta(1-a)}.
\end{equation*}
\end{lem}
\begin{lem}\cite[Theorem 1]{chen2012distributed}\label{yinli123}
If there exists a positive constant $\nu$ such that $\sup_{t\in[0,\infty)}||\dot{v}_i(t)||_{\infty}\leq\nu$, for any $i\in\mathscr{V}$, then under the following algorithm,
\begin{equation*}
\begin{aligned}
\dot{\beta}_i&=-\alpha\sum_{j\in \mathscr{N}_i}\sign(\eta_i-\eta_j),\\
\eta_i&=\beta_i+||v_i||,
\end{aligned}
\end{equation*}
we have $\big|\eta_i-\frac{1}{N}\sum_{i=1}^N||v_i||\big|\rightarrow0$ in finite time, where $\alpha>\nu$.
\end{lem}
\subsection{Matrix Transformation}
Some useful lemmas on matrix transformation are given in this subsection.
\begin{lem}\cite[Lemma 3]{seuret2013wirtinger}\label{yinli32}
For given $n, m\in\mathbb{Z}_+$, $\chi\in(0,1)$, a positive definite matrix $Z\in\mathbb{R}^{n \times n}$, and two matrices $W_1,W_2\in\mathbb{R}^{n \times m}$, define the function $f(\chi, Z)$ for any vector $x\in \mathbb{R}^m$ as
$$
f(\chi, Z)=\frac{1}{\chi} x^{\rm T} W_1^{\rm T} Z W_1 x+\frac{1}{1-\chi} x^{\rm T} W_2^{\rm T} Z W_2 x .
$$
If there exists a matrix $S \in \mathbb{R}^{n \times n}$ such that $\left[\begin{array}{cc}Z & S \\ S^{\rm T} & Z\end{array}\right]\succ0$, then the following inequality holds
$$
\min _{\chi \in(0,1)} f(\chi, Z) \geq x^{\rm T}\left[\begin{array}{l}
W_1 \\
W_2
\end{array}\right]^{\rm T}\left[\begin{array}{ll}
Z & S \\
S^{\rm T} & Z
\end{array}\right]\left[\begin{array}{l}
W_1 \\
W_2
\end{array}\right] x.
$$
\end{lem}

\begin{lem}\cite[Theorem 4.6]{hkk2002}\label{yinli33}
For any Hurwitz matrix $\mathcal{H}$, there exist positive definite matrices $\mathcal{P}$, and $\mathcal{Q}$ such that
\begin{equation*}
\mathcal{P}\mathcal{H}+\mathcal{H}^{\rm T}\mathcal{P}=-\mathcal{Q}.
\end{equation*}
\end{lem}
\begin{lem}\cite[Lemma 3.3]{zuo2016distributed}\label{yinli34}
For any constants $h_1,h_2,\ldots,$ $h_n\geq0$ and $\varepsilon\in(0,1]$, we have $\sum_{i=1}^nh_i^{\varepsilon}\geq(\sum_{i=1}^nh_i)^{\varepsilon}.$
\end{lem}
\subsection{Supertwisting-based ISMC Scheme}\label{supertwist}
This subsection introduces the supertwisting-based control scheme, which serves as the theoretical foundation for the compensating mechanism proposed in this paper.

While traditional first-order Sliding Mode Control (SMC) provides robust attenuation against external disturbances, it necessitates a discontinuous, high-frequency switching control signal (typically a signum function, $\operatorname{sign}(\sigma)$) to maintain the system state on the sliding surface. This high-frequency switching induces ``chattering'', severe high-frequency oscillations that can seriously degrade physical actuators. To alleviate chattering while ensuring finite-time disturbance rejection, the supertwisting-based ISMC was proposed in \cite{levant1993sliding}. As a second-order SMC scheme, it yields a continuous control signal by embedding the discontinuous switching function within an integral term. A standard supertwisting-based ISMC scheme is governed by
\begin{equation}\label{twist}\begin{aligned} & \dot{\sigma}_1=-\kappa_1 \sig^{\frac{1}{2}}(\sigma_1)+\sigma_2+d_1(t), \\ & \dot{\sigma}_2=-\kappa_2 \operatorname{sign}(\sigma_1)+d_2(t), \end{aligned}\end{equation}
where $\sigma_1$ and $\sigma_2$ denote the scalar variables, $\kappa_1$ and $\kappa_2$ are control gains to be designed, and $d_1(t)$ and $d_2(t)$ are external perturbations. By appropriately tuning the gains $\kappa_1$ and $\kappa_2$, the system \eqref{twist} ensures finite-time compensation for bounded perturbations, provided $d_1(t)=0$ and $|d_2(t)| \leq \bar{d}$.

Although extensively applied to consensus and optimization problems \cite{yu2015finite,feng2019finite}, the standard supertwisting-based ISMC scheme is inadequate for the distributed NE seeking problem addressed herein. Specifically, the coexistence of external disturbances and state-dependent uncertain dynamics may violate the global boundedness of the perturbations required by this scheme, which motivates the design of a novel composite compensating mechanism in the subsequent sections.
\subsection{Problem Description}
This paper explores a noncooperative game involving $N$ cyber-physical players communicating through an undirected graph $\mathscr{G:=(V,E)}$. In the cyber layer, each player tends to minimize its own cost function $f_{i}\left(x_{i}, x_{-i}\right)$ where $x_{i}\in\mathbb{R}^n$ is player $i$'s action, and $x_{-i}:=\text{col}(x_{1},\ldots,x_{i-1},x_{i+1},\ldots,x_{N})$. The ultimate task for player $i$ is given as
\begin{equation}\label{na12}\begin{array}{ll}
\underset{x_{i}\in\mathbb{R}^n}{\mbox{minimize  }}  &f_{i}(x_{i},x_{-i}).\
\end{array}\end{equation}While \eqref{na12} defines players' ideal objectives, the actual action execution in cyber-physical systems relies on physical entities, such as the movement of mobile sensor vehicles, where a second-order system involving velocity and acceleration information is considered. Furthermore, during the process of task achievement, the physical movement of player $ i $ will inevitably encounter external environmental disturbances (e.g., wind gusts, unknown road friction) and unmodeled mechanical dynamics. Thus, the physical execution of player $i$ can be formulated as the following disturbed second-order system
	\begin{align}\label{erjie}
		\dot x_{i}(t)&=v_i(t),\\
        \dot v_i(t)&=u_i+\omega_i(t)+\varrho_i(\mathbf{x})\nonumber,
	\end{align}
where the action $x_i$ corresponds to player $i$'s position, $v_i\in\mathbb{R}^n$ denotes player $i$'s velocity, $u_i$ is player $i$'s control input, $\mathbf{x}=\col(x_1,x_2,\ldots,x_N)$, $\omega_i(t)$ and $\varrho_i(\mathbf{x})$ are bounded disturbance and uncertain dynamics, respectively.

For the second-order system \eqref{erjie}, the ultimate goal of this paper is not only to steer players' positions $\mathbf{x}$ to the NE, but to drive players' velocities $v:=\col(v_1,v_2,\ldots,v_N)$ to $\mathbf{0}_{Nn}$.
\begin{defn}\cite[Definition 1]{hu2022distributed}
An action profile $x^{*}:=\text{col}(x_{1}^{*},\ldots,x_{N}^{*})\in \mathbb{R}^{Nn}$ is called an NE of the noncooperative game \eqref{na12} under the condition that
\begin{equation*}
			f_{i}(x_{i}^{*},x_{-i}^{*})\leq f_{i}(x_{i},x_{-i}^{*}), ~\forall x_i\in\mathbb{R}^{n},~\forall i\in\mathscr{V},
		\end{equation*}
where $x_{-i}^*:=\text{col}(x_{1}^*,\ldots,x_{i-1}^*,x_{i+1}^*,\ldots,x_{N}^*).$
\end{defn}

		In other words, an NE signifies a situation where no player can decrease its cost by solely altering its own action.

To implement the forthcoming convergence analysis, the following assumptions and definition are required.
\begin{assum}\label{jiashe1}
The communication graph $\mathscr{G}$ is undirected and connected.
\end{assum}
\begin{defn}
The pseudo-gradient for the cost functions of all players in stack form is defined as follows,
		$$F(\mathbf{x}):=\text{col}(\nabla_{1}f_{1}(\mathbf{x}),\nabla_{2}f_{2}(\mathbf{x}),\ldots,\nabla_{N}f_{N}(\mathbf{x})),$$
in which $\nabla_{i}f_{i}(\mathbf{x}):=\frac{\partial f_{i}}{\partial x_{i}}(\mathbf{x})$.
		
	\end{defn}
\begin{assum}\label{jiashe22} The following declarations hold.
\begin{itemize}
\item For each $i\in\mathscr{V}$, $f_i(\mathbf{x})$ is continuously differentiable and convex with respect to $x_i$ for any fixed $x_{-i}$, as well as $\nabla_i f_i(\mathbf{x})$ is globally Lipschitz with modulus $l_i$.
\item $F:\mathbb{R}^{Nn}\to\mathbb{R}^{Nn}$ is a $\mu$-strongly monotone function.
\end{itemize}
\end{assum}
\begin{rem}
Under Assumption \ref{jiashe22}, we can obtain that there exists a unique NE $x^*$ satisfying $F(x^*)=\mathbf{0}_{Nn}$ based on \cite[Proposition 12.3]{facchinei201012}.
\end{rem}
\begin{assum}\label{jiashe24}
For each $i\in\mathscr{V}$, $\varrho_i(\mathbf{x})$ is continuously differentiable, bounded and $\|\nabla \mathbf{\varrho}(\mathbf{x})\|\leq\tilde{g}$, where $\varrho(\mathbf{x}):=\col(\varrho_1(\mathbf{x}),\varrho_2(\mathbf{x}),\ldots,\varrho_N(\mathbf{x}))$, for some $\tilde{g}>0$.
\end{assum}
\begin{assum}\label{jiashe25}
For each $i\in\mathscr{V}$, $\omega_i(t)$ and its time derivative $\dot{\omega}_i$ are bounded by known constants.
\end{assum}
\begin{rem}
In this paper, we relax the assumption on the boundedness of the second-order derivatives of the external disturbances compared with \cite{wang2020distributed}. Moreover, many types of practical disturbances satisfy this relaxed assumption, such as sinusoidal, constant, and ramp disturbances.
\end{rem}
\section{Distributed NE Seeking without Semi-Markov Switching}\label{sec3}
This section designs a distributed algorithm to seek the NE of game \eqref{na12}. Since evaluating the cost $f_i(x_i,x_{-i})$ over a local network requires the unobservable actions of others ($x_{-i}$), each player must employ an estimation method. Inspired by the leader-follower protocol, player $i$ maintains an estimator $y^i:=\col(y_1^i,\ldots,y_N^i)\in\mathbb{R}^{Nn}$, where $y_i^i=x_i$ and $y_j^i$ estimates player $j$'s action for $j\neq i$. Based on this estimation, the operational flow of Algorithm 1 is summarized in TABLE \ref{a1}.
\begin{table}[!htbp]
\vspace{-0.2cm}
\small
    \centering
\setlength{\abovedisplayskip}{0pt}
\setlength{\belowdisplayskip}{0pt}
\renewcommand\arraystretch{1}
\caption{The operational flow of Algorithm 1\label{a1}}
    \begin{tabular}{p{0.93\columnwidth}}
        \hline
        \textbf{Algorithm 1:} \color{black}\textbf{Distributed robust NE seeking}\\
        \hline
        \textbf{Initialization:}\\
        \quad For $i\in \mathscr{V}$, initial states $x_i(0)$, $v_i(0)$, $y^i(0)$.\\
        \quad Initial compensators: $\beta_i(0) \!=\! 0$, $s_i(0) \!=\! \mathbf{0}_n$, $\phi_i(0) \!=\! \mathbf{0}_n$.\\
        \quad Set parameters: $0<k_{1}<\frac{\mu}{\max_{i\in\mathscr{V}}\{l_i\}^2}$, $k_{i 2}>0$, \\\quad\quad\quad\quad\quad~~\quad\quad$k_{i 4}>0$, $k_{i3}>\|\dot{\omega}_i\|_{\infty}$, $\alpha$, $\tilde{g}$.\\
        \textbf{Main Loop (Algorithm Execution):}\\
        \quad \textbf{while} $t \ge 0$ \textbf{do}\\
        \quad\quad \textbf{1. State Measurement \& Communication:}\\
        \quad\quad\quad Measure local state $x_i(t)$ and velocity $v_i(t)$.\\
        \quad\quad\quad Receive $y^m(t)$ and $\eta_m(t)$ from neighbors $m \in \mathscr{N}_i$.\\
        \quad\quad\quad Compute $u_i^0(t) = -k_1 \nabla_i f_i(y^i(t)) - v_i(t)$.\\
        \quad\quad \textbf{2. Compensator Variable Update:}\\
        \quad\quad\quad $s_i(t) = v_i(t) - v_i(0) - \int_0^t u_i^0(\rho)\text{d}\rho$.\\
        \quad\quad\quad $\eta_i(t) = \beta_i(t) + \|v_i(t)\|$.\\
        \quad\quad\quad $\dot{\beta}_i(t) = -\alpha \sum_{j \in \mathscr{N}_i} \operatorname{sign}(\eta_i(t) - \eta_j(t))$.\\
        \quad\quad\quad $\dot{\phi}_i(t) = -k_{i3}\operatorname{sgn}(s_i(t)) - \tilde{g}N\eta_i(t)\operatorname{sgn}(s_i(t))$.\\
        \quad\quad\quad $u_{i}^{r}(t)=-k_{i 2}\operatorname{sig}^{\frac{1}{2}}\left(s_{i}(t)\right)+\phi_{i}(t)$.\\
        \quad\quad \textbf{3. Estimation \& Control Law Computation:}\\
        \quad\quad\quad $u_i(t) = u_i^r(t)+u_i^0(t)$.\\
        \quad\quad\quad Update action estimator:\\
        \quad\quad\quad $\varepsilon\dot{y}_{j}^i(t) = -k_{i 4}\big(\sum_{m=1}^{N} a_{i m}(y^{i}_j(t)-y^{m}_j(t)) $\\
        \quad\quad\quad\quad\quad\quad~~~$+a_{i j}(y^{i}_j(t)-x_{j}(t))\big)$.\\
        \quad\quad \textbf{4. State Evolution:}\\
        \quad\quad\quad Apply control input $u_i(t)$ to the physical plant:\\
        \quad\quad\quad $\dot{x}_i(t) = v_i(t)$, \ \ $\dot{v}_i(t) = \omega_i(t) + \varrho_i(\mathbf{x}(t)) + u_i(t)$.\\
        \quad \textbf{end while}\\
        \hline
    \end{tabular}
    \vspace{-0.2cm}
\end{table}

To clearly explain the structure of Algorithm 1, we decompose it into three main parts in the following.
\begin{itemize}
  \item Second-order NE seeking
  \begin{equation}\label{nashxunzhao}
  \left\{\begin{aligned}
  \dot{x}_{i}&=v_{i}, \\
\dot{v}_{i}&=\omega_{i}+\varrho_i(\mathbf{x})+u_{i},\\
u_i&=u_i^0+u_i^r,\\
u_{i}^{0}&=-k_{1} \nabla_{i} f_{i}\left(y^{i}\right)-v_i.
\end{aligned}\right.
\end{equation}
  \item Finite-time compensating mechanism for disturbances and uncertain dynamics
  \begin{equation}\label{buchangjizhi}
  \left\{\begin{aligned}
  u_{i}^{r}&=-k_{i 2} \operatorname{sig}^{\frac{1}{2}}\left(s_{i}\right)+\phi_{i},\\
\dot{\phi}_{i}&=-k_{i3} \operatorname{sgn}\left(s_{i}\right)-\tilde{g}N\eta_i\sgn(s_i),\\
s_{i}&=v_{i}-v_{i}(0)-\int_{0}^{t} u_{i}^{0}(\rho) d \rho,\\
\dot{\beta}_i&=-\alpha\sum_{j\in \mathscr{N}_i}\sign(\eta_i-\eta_j),\\
\eta_i&=\beta_i+||v_i||.
\end{aligned}\right.
\end{equation}
  \item Action estimation
    \begin{equation}\label{guji1}
  \begin{aligned}
 \varepsilon\dot{y}_{j}^i=-k_{i 4}\left(\sum_{m=1}^{N} a_{i m}\left(y^{i}_j-y^{m}_j\right)+a_{i j}\left(y^{i}_j-x_{j}\right)\right).
\end{aligned}
\end{equation}
\end{itemize}
\begin{rem}
The NE seeking algorithm 1 is proposed by adopting supertwisting-based ISMC scheme in $s_i$ and $\phi_i$ combined with average tracking protocol (see Lemma \ref{yinli123}) in $\eta_i$ and $\beta_i$, which can reject disturbances and uncertain dynamics in finite time. Specifically, the control input $u_i$ in Algorithm 1 has two parts: the second-order NE seeking controller $u_i^0$ and ISMC controller $u_i^r$.
\end{rem}

Since Algorithm 1 has a two-time-scale structure, to facilitate the subsequent stability analysis, singular perturbation technique is employed to decouple and analyze the distinct time scales. Note that within this technique, Algorithm 1 is typically referred to as the corresponding original system. In this system composed of \eqref{nashxunzhao}, \eqref{buchangjizhi} and \eqref{guji1}, the physical states $x_i,~v_i$ and the compensator variables $\beta_i,~\phi_i$ act as the slow variables evolving in the normal time scale $t$, whereas the action estimation state $y^i$ acts as the fast variable modulated by the small singular perturbation parameter $\varepsilon$.

First, to derive the reduced system (which describes the slow dynamics), we artificially set $\varepsilon = 0$ in the fast dynamics \eqref{guji1}. This yields the following equation$$\mathbf{0}_n = -k_{i 4}\left(\sum_{m=1}^{N} a_{i m}\left(y^{i}_j-y^{m}_j\right)+a_{i j}\left(y^{i}_j-x_{j}\right)\right).$$By solving this equation, we obtain the unique isolated root, which represents the quasi-steady state of the fast variable, as $\mathbf{y}_{\sigma}:=\col(y^1_{\sigma},y^2_{\sigma},\ldots,y^N_{\sigma})=\mathbf{1}_N\otimes\mathbf{x}$. Substituting this quasi-steady state $\mathbf{y}_{\sigma}$ back into the slow dynamics \eqref{nashxunzhao} and \eqref{buchangjizhi}, we obtain the reduced system that governs the original system's behavior on the slow manifold\begin{equation}\label{000}\left\{\begin{aligned}\dot{x}_{i}&=v_{i}, \\
\dot{v}_{i}&=\omega_{i}+\varrho_i(\mathbf{x})+u_{i}^r-k_{1} \nabla_{i} f_{i}\left(\mathbf{x}\right)-v_i,\\
\dot{\phi}_{i}&=-k_{i3} \operatorname{sgn}\left(s_{i}\right)-\tilde{g}N\eta_i\sgn(s_i),\\
\dot{\beta}_i&=-\alpha\sum_{j\in \mathscr{N}_i}\sign(\eta_i-\eta_j).
\end{aligned}\right.\end{equation}

Next, to derive the boundary-layer system (which describes the fast dynamics), we introduce the stretched fast time scale $\tau = \frac{t}{\varepsilon}$. In this fast time scale $\tau$, taking the slow variable $x_i$ as an example, its derivative with respect to $\tau$ becomes $\frac{d x_i}{d\tau} = \varepsilon v_i$. Consequently, as $\varepsilon \to 0$, all the slow variables act as frozen constants (e.g., $\frac{d x_i}{d\tau} = 0$). Evaluating \eqref{guji1} in the $\tau$ time domain with frozen slow variables yields the boundary-layer system\begin{equation}\label{bianjie1}\begin{aligned}\frac{\text{d}{y}_{j}^i}{\text{d}\tau}=-k_{i 4}\left(\sum_{m=1}^{N} a_{i m}\left(y^{i}_j-y^{m}_j\right)+a_{i j}\left(y^{i}_j-x_{j}\right)\right).\end{aligned}\end{equation}

Note that through system decomposition, the variables in the reduced system and the boundary-layer system are distinct from those in the original system. However, for the sake of convenience, we will not define new variables to distinguish them. With the original system explicitly linked to the reduced system \eqref{000} and the boundary-layer system \eqref{bianjie1}, the rest of the analysis can be methodically divided into three parts: (i) Analyze the stabilities of the states of the boundary-layer system and the reduced system. (ii) Show the closeness of solutions between the limit systems and the original system. (iii) Give the stability results of the state of the original system.

We start by dealing with disturbances and uncertain dynamics. In the next two lemmas, we first prove that the supertwisting-based ISMC scheme with average tracking protocol is able to compensate for the influence of disturbances and uncertain dynamics in finite time.
\begin{lem}\label{yinli42}\label{yinlili1}
Under Assumptions \ref{jiashe22}-\ref{jiashe25}, if $k_{i3}>\|\dot{\omega}_i\|_{\infty}$ for $i\in\mathscr{V}$, then the influence of disturbances and uncertain dynamics will be eliminated in finite time by the reduced system \eqref{000}, i.e., there exists a positive constant $T_{\max}$ such that $\lim_{t\rightarrow T_{\max}}\|u_{i}^{r}+\omega_i+\varrho_{i}(\mathbf{x})\|=0$.
\end{lem}
\begin{proof} By taking the time derivative of $s_i$, the compensating part in the reduced system \eqref{000} transforms as
\begin{equation*}
\left\{\begin{aligned}
\dot{s}_{i} &=u_{i}^{r}+\omega_i+\varrho_{i}(\mathbf{x})\\
&=-k_{i2} \sig^{\frac{1}{2}}(s_{i})+\varphi_{i}, \\
\dot{\varphi}_{i} &=-k_{i3} \operatorname{sgn}\left(s_{i}\right)-\tilde{g}N\eta_i\sgn(s_i)+\dot{\omega}_{i}(t)+\sum_{i=1}^N\frac{\partial \varrho_{i}(\mathbf{x})}{\partial x_i}v_i,\\
\dot{\beta}_i&=-\alpha\sum_{j\in \mathscr{N}_i}\sign(\eta_i-\eta_j),\\
\varphi_{i} &=\phi_{i}+\omega_i+\varrho_{i}(\mathbf{x}),\\
\eta_i&=\beta_i+||v_i||.\\
\end{aligned}\right.
\end{equation*}
Define a stack variable $\mathfrak{A}_{ij}=\col(\sig^{\frac{1}{2}}(s_{ij}),\varphi_{ij})$, where $s_i=\col\{s_{i1},s_{i2},\ldots,s_{in}\}$ and $\varphi_i=\col\{\varphi_{i1},\varphi_{i2},\ldots,\varphi_{in}\}$. Choose the candidate Lyapunov function $$V=\sum_{i=1}^N\sum_{j=1}^nV_{ij}=\sum_{i=1}^N\sum_{j=1}^n\mathfrak{A}_{ij}^{\rm T}\mathcal{P}_{ij}\mathfrak{A}_{ij},$$ where $\mathcal{P}_{ij}\in\mathbb{R}^{2\times2}$ is a positive definite matrix for $i\in\mathscr{V}$ and $j\in\{1,2,\ldots,n\}$. Note that with the part $\sig^{\frac{1}{2}}(s_{ij})$, $V$ is absolutely continuous but not locally Lipschitz on the set $\Upsilon=\{(s_{ij},\phi_{ij})\in\mathbb{R}^{2}|s_{ij}=0\}$. Thus, classical Lyapunov theorem does not work, which needs the continuous differentiability or at least locally Lipschitz continuity of $V$. Based on the proof of Zubov theorem in \cite[Theorem 1]{moreno2012strict}, a merely continuous $V$ is suitable for stability analysis. Moreover, it can be easily obtained that $V$ is an absolutely continuous function with respect to $t$. Therefore, its time derivative is defined almost everywhere.

The upper-right Dini derivative of $\mathfrak{A}_{ij}$ is calculated as
\begin{equation*}
\begin{array}{c}
\text{D}^+\mathfrak{A}_{ij}=\frac{1}{2}\left|s_{ij}\right|^{-\frac{1}{2}} \left[\begin{array}{c}
-k_{i2} \operatorname{sig}^{\frac{1}{2}}\left(s_{ij}\right)+\varphi_{ij} \\
\mathfrak{P}_{ij}
\end{array}\right],\\
\end{array}
\end{equation*}
where $\mathfrak{P}_{ij}=-2(k_{i3}+\tilde{g}N\eta_i-(\dot{\omega}_{ij}+\sum_{i=1}^N\frac{\partial \varrho_{i}(\mathbf{x})}{\partial x_{ij}}v_{ij})$ $\times\operatorname{sign}(s_{ij})) \operatorname{sig} ^{\frac{1}{2}}(s_{ij})$.
Taking the derivative of $V$ along the direction of the compensating part, we have
\begin{equation*}
\begin{aligned}
\text{D}^+V=&\sum_{i=1}^{N}\sum_{j=1}^n\left|s_{ij}\right|^{-\frac{1}{2}} \mathfrak{A}_{ij}^{\rm T}\left[\mathcal{R}_{ij}^{\rm T} \mathcal{P}_{ij}+\mathcal{P}_{ij} \mathcal{R}_{ij}\right]\mathfrak{A}_{ij},\\
\end{aligned}
\end{equation*}
where $$\mathcal{R}_{ij}=\left[\begin{array}{cc}
-\frac{1}{2} k_{i 2} & \frac{1}{2} \\
\mathfrak{R}_{ij} & 0
\end{array}\right],$$ in which $\mathfrak{R}_{ij}=-\left[k_{i 3}-\dot{\omega}_{ij} \operatorname{sign}\left(s_{ij}\right)\right]+\sum_{i=1}^N\frac{\partial \varrho_{i}(\mathbf{x})}{\partial x_{ij}}v_{ij}$ $\times\sign(s_{ij})-\tilde{g}N\eta_i$. Based on the boundedness of $\frac{\partial \varrho_i}{\partial x_{ij}}$ and Lemma \ref{yinli123}, we can get that there exists a $T > 0$ such that for all $t>T$, the following formula holds
\begin{equation*}
\begin{aligned}
-2\tilde{g}\sum_{i=1}^N||v_i||\leq&\sum_{i=1}^N\frac{\partial \varrho_{i}(\mathbf{x})}{\partial x_{ij}}v_{ij}\sign(s_{ij})-\tilde{g}N\eta_i\\
\leq&\tilde{g}\sum_{i=1}^N \|v_i\|-\tilde{g}\sum_{i=1}^N||v_i||=0.
\end{aligned}
\end{equation*}What is more, when $k_{i2}>0$, and $k_{i3}>\|\dot{\omega}_i\|_{\infty}$, based on Assumption \ref{jiashe25}, we can get $\mathcal{R}_{ij}$ is Hurwitz.

Since $\mathcal{R}_{ij}$ is Hurwitz, based on Lemma \ref{yinli33} and Theorem 2 in \cite{moreno2012strict}, there exists a unique positive definite matrix $\mathcal{P}_{ij}$ such that $\mathcal{R}_{ij}^{\rm T} \mathcal{P}_{ij}+\mathcal{P}_{ij} \mathcal{R}_{ij}=-\mathcal{Q}_{ij}$ for each positive definite matrix $\mathcal{Q}_{ij}$. Then, we obtain that
\begin{equation*}
\text{D}^+V\leq-\sum_{i=1}^N\sum_{j=1}^n\left|s_{ij}\right|^{-\frac{1}{2}}\mathfrak{A}^{\rm T}_{ij}\mathcal{Q}_{ij}\mathfrak{A}_{ij}\leq0.
\end{equation*}Based on $|s_{ij}|^{\frac{1}{2}}\leq|\mathfrak{A}_{ij}|\leq\lambda_{\min}(\mathcal{P}_{ij})^{\frac{1}{2}}V_{ij}^{\frac{1}{2}}$, one derives
\begin{equation*}
\begin{aligned}
\text{D}^+V&\leq-\sum_{i=1}^N\sum_{j=1}^n\lambda_{\min}(\mathcal{P}_{ij})V_{ij}^{-\frac{1}{2}}\frac{\lambda_{\min}(\mathcal{Q}_{ij})}{\lambda_{\max}(\mathcal{P}_{ij})}V_{ij}\\
&\leq-c_0\sum_{i=1}^N\sum_{j=1}^nV_{ij}^{\frac{1}{2}}\leq-c_0(\sum_{i=1}^N\sum_{j=1}^nV_{ij})^{\frac{1}{2}}=-c_0V^{\frac{1}{2}},
\end{aligned}
\end{equation*}where $c_0=\min\{\frac{\lambda_{\min}(\mathcal{P}_{ij})\lambda_{\min}(\mathcal{Q}_{ij})}{\lambda_{\max}(\mathcal{P}_{ij})}\}$ and the third inequality follows from Lemma \ref{yinli34}. Then, based on the definition of the variable $s_i$, we can get $s_i(0)=\mathbf{0}_n$. Therefore, according to Lemma \ref{yinli35}, $s_{ij}(t)$ is locally convergent to 0 in finite time estimated by $T_{\max}=\frac{2V^\frac{1}{2}(0)}{c_0}$. What is more, it also means that $\lim_{t\rightarrow T_{\max}}\|u_{i}^{r}+\omega_i+\varrho_{i}(\mathbf{x})\|=0$. The influence of disturbances and uncertain dynamics is eliminated in finite time.
\end{proof}
\begin{rem}
From Lemma \ref{yinlili1}, the ISMC controller $u_i^r$ ensures that the NE seeking part in the reduced system \eqref{000} reside on the sliding manifold in finite time with disturbances and uncertain dynamics rejection. On the basis of the finite-time convergence result in Lemma \ref{yinli42}, when $t\geq T_{\max}$, the reduced system \eqref{000} can be reduced into
\begin{equation}\label{tuihua}
 \left\{\begin{aligned}
  \dot{x}_{i}&=v_{i}, \\
\dot{v}_{i}&=-k_{1} \nabla_{i} f_{i}\left(\mathbf{x}\right)-v_i,\\
\dot{s}_i&=0.
\end{aligned}\right.
\end{equation}
\end{rem}

Then, the stability result of the NE seeking part in the reduced system \eqref{000} is given below.
\begin{thm}\label{dinglili1}
Under Assumption \ref{jiashe22}, when $t\geq T_{\max}$, for $0<k_1<\frac{\mu}{\max_{i\in\mathscr{V}}\{l_i\}^2}$, the state of system \eqref{tuihua} is exponentially stable and the state $\mathbf{x}$ will exponentially converge to the NE of the noncooperative game \eqref{na12}.
\end{thm}
\begin{proof} Transforming system \eqref{tuihua} into the stack form, we have
\begin{equation}\label{010}
  \left\{\begin{aligned}
  \dot{\mathbf{x}}&=v, \\
\dot{v}&=-k_{ 1} F(\mathbf{x})-v.
\end{aligned}\right.
\end{equation}
Denote $x^*$ as the NE of the noncooperative game \eqref{na12}. We define $$V=V_1+V_2,$$ where $V_1=\frac{1}{2}(\mathbf{x}+v-x^*)^{\rm T}(\mathbf{x}+v-x^*)$, $V_2=\frac{1}{2}v^{\rm T}v$. Differentiating $V_1$ and $V_2$ along system \eqref{010}, one derives
\begin{equation*}
\begin{aligned}
\dot{V}_1=&(\mathbf{x}+v-x^*)^{\rm T}\big(v-k_1(F(\mathbf{x})-F(x^*))-v\big)\\
=&-k_1(\mathbf{x}-x^*)^{\rm T}(F(\mathbf{x})-F(x^*))-k_1v^{\rm T}(F(\mathbf{x})-F(x^*)),\\
\end{aligned}
\end{equation*}
and
\begin{equation*}
\begin{aligned}
\dot{V}_2=&v^{\rm T}(-k_1(F(\mathbf{x})-F(x^*))-v)\\
=&-\|v\|^2-k_1v^{\rm T}(F(\mathbf{x})-F(x^*)).
\end{aligned}
\end{equation*}
Then, we can obtain
\begin{equation*}
\begin{aligned}
\dot{V}=&-k_1(\mathbf{x}-x^*)^{\rm T}(F(\mathbf{x})-F(x^*))-\|v\|^2\\
&-2k_1v^{\rm T}(F(\mathbf{x})-F(x^*))\\
\leq&-k_1\mu\|\mathbf{x}-x^*\|^2-\|v\|^2+2k_1\max_{i\in\mathscr{V}}\{l_i\}\|v\|\|\mathbf{x}-x^*\|,
\end{aligned}
\end{equation*}
where the last inequality is based on the $\mu$-strong monotonicity of $F$ and $l_i$-Lipschitz continuity of $\nabla_if_i(\mathbf{x})$.
Choosing $0<k_1<\frac{\mu}{\max_{i\in\mathscr{V}}\{l_i\}^2}$, one can get
\begin{equation*}
\dot{V}\leq-\lambda_{\min}(\mathfrak{B})\|\mathfrak{Q}\|^2\leq-\frac{2\lambda_{\min}(\mathfrak{B})}{3}V,
\end{equation*}
where matrix $\mathfrak{B}:=\left[\begin{array}{cc}
k_1\mu & -k_1\max_{i\in\mathscr{V}}\{l_i\} \\
-k_1\max_{i\in\mathscr{V}}\{l_i\} & 1
\end{array}\right]$ is positive definite, and $\mathfrak{Q}:=\col(\mathbf{x}-x^*,v)$.

As a result of the above analysis, we can derive that the state of system \eqref{tuihua} is exponentially stable and the state $\mathbf{x}$ will exponentially converge to the NE of the noncooperative game \eqref{na12} under the second-order dynamics.
\end{proof}

The stability result of the state of the boundary-layer system \eqref{bianjie1} is given by the following theorem.
\begin{thm}\label{dinglili2}
Under Assumption \ref{jiashe1}, when $t\geq T_{\max}$, the state of the boundary-layer system \eqref{bianjie1} is exponentially stable.
\end{thm}
\begin{proof} First of all, the boundary-layer system \eqref{bianjie1} can be restated integrally as
$$\frac{\mathrm{d}\boldsymbol{y}}{\mathrm{d} \tau}=-k_4((L\otimes \mathbf{I}_{N}+\mathfrak{M})\otimes\mathbf{I}_{n})\left(\boldsymbol{y}-\left(\mathbf{1}_{N} \otimes \mathbf{x}\right)\right),$$
where $\boldsymbol{y}=\col(y^1,y^2,\ldots,y^N)$ and $\mathfrak{M}=\diag_{i,j\in\mathscr{V}}\{a_{ij}\}$, in which $\diag_{i,j\in\mathscr{V}}\{a_{ij}\}$ denotes a diagonal matrix whose diagonal elements are $a_{11},a_{12},\ldots,a_{1N},a_{21},a_{22},\ldots,a_{2N},$ $\ldots,a_{NN}$, successively.

For convenience, denote estimation error $\boldsymbol{e}=\boldsymbol{y}-\left(\mathbf{1}_{N} \otimes \mathbf{x}\right)$. It should be noted that based on the singular perturbation theory in \cite[Theorem 11.1]{hkk2002}, after making the system decomposition, the variable $\mathbf{x}$ is normally treated as a fixed vector. Thus, the boundary-layer system \eqref{bianjie1} can be transformed as
\begin{equation}\label{wuchaxi}\frac{\mathrm{d}\boldsymbol{e}}{\mathrm{d} \tau}=-k_4((L\otimes \mathbf{I}_{N}+\mathfrak{M})\otimes\mathbf{I}_{n})\boldsymbol{e}.
\end{equation}Due to the fact that $-((L\otimes \mathbf{I}_{N}+\mathfrak{M})\otimes\mathbf{I}_{n})$ is Hurwitz, based on Lemma \ref{yinli33} and Theorem 2 in \cite{moreno2012strict}, there exists a unique positive definite matrix $\boldsymbol{H}$ such that $-((L\otimes \mathbf{I}_{N}+\mathfrak{M})\otimes\mathbf{I}_{n})^{\rm T}\boldsymbol{H}-\boldsymbol{H}((L\otimes \mathbf{I}_{N}+\mathfrak{M})\otimes\mathbf{I}_{n})=-\boldsymbol{Q}$ for positive definite matrix $\boldsymbol{Q}$. Define candidate Lyapunov function
\begin{equation*}
\tilde{V}=\boldsymbol{e}^{\rm T}\boldsymbol{H}\boldsymbol{e}.
\end{equation*}
Taking the derivative of $\tilde{V}$ along the state of system \eqref{wuchaxi}, we can get
\begin{equation*}
\begin{aligned}
\frac{\mathrm{d}\tilde{V}}{\mathrm{d} \tau}=&-k_4(\boldsymbol{e}^{\rm T}((L\otimes \mathbf{I}_{N}+\mathfrak{M})\otimes\mathbf{I}_{n})^{\rm T}\boldsymbol{H}\\
&+\boldsymbol{H}((L\otimes \mathbf{I}_{N}+\mathfrak{M})\otimes\mathbf{I}_{n})\boldsymbol{e})\\
=&-k_4\boldsymbol{e}^{\rm T}\boldsymbol{Q}\boldsymbol{e}\leq-k_4\lambda_{\min}(\boldsymbol{Q})||\boldsymbol{e}||^2\leq-k_4\frac{\lambda_{\min}(\boldsymbol{Q})}{\lambda_{\max}(\boldsymbol{H})}\tilde{V}.
\end{aligned}
\end{equation*}

As a result of the above analysis, we can derive that the state of the boundary-layer system \eqref{bianjie1} is exponentially stable and $||\boldsymbol{e}||\rightarrow0$ exponentially. \end{proof}

Now, we can finally conclude the stability result of the state of the original system based on the following theorem.
\begin{thm}\label{dingliyuan}
Suppose that Assumptions \ref{jiashe1}-\ref{jiashe25} hold. When $t\geq T_{\max}$, if $0<k_1<\frac{\mu}{\max_{i\in\mathscr{V}}\{l_i\}^2}$ and $k_{i3}>\|\dot{\omega}_i\|_{\infty}$ for $i\in\mathscr{V}$, there exists $\varepsilon^*>0$ such that for each $\varepsilon \in\left(0, \varepsilon^*\right]$, the state of Algorithm 1 is exponentially stable and the state $\mathbf{x}$ will exponentially converge to the NE of the noncooperative game \eqref{na12}.
\end{thm}
\begin{proof} According to Lemma \ref{yinlili1}, when $t\geq T_{\max}$, the compensating part rejects the disturbances and uncertain dynamics, allowing the reduced system \eqref{000} to strictly follow the simplified dynamics \eqref{tuihua}. Then, based on Theorems \ref{dinglili1}-\ref{dinglili2}, the states of the boundary-layer system \eqref{bianjie1} and the system \eqref{tuihua} are both exponentially stable. Thus, according to Theorem 11.4 of \cite{hkk2002}, the proof of Theorem \ref{dingliyuan} is completed.\end{proof}

\section{Distributed NE Seeking with Semi-Markov Switching}\label{sec4}
While Algorithm 1 effectively estimates actions, its reliance on continuous communication consumes excessive resources, hindering practical applications such as Unmanned Aerial Vehicle (UAV) coordination under power constraints \cite{kownacki2019adaptation} and vehicle connectivity control under environmental interference \cite{stankovic2011distributed}. To alleviate this, an event-triggered mechanism is essential. Furthermore, real-world communication topologies frequently fluctuate due to uncertain factors like signal shielding and connection interruptions. Consequently, Algorithm 2 considers NE seeking under event-triggered mechanism and semi-Markov switching topologies.
\begin{table}[!h]
\small
    \centering
\setlength{\abovedisplayskip}{0pt}
\setlength{\belowdisplayskip}{0pt}
\renewcommand\arraystretch{1}
\caption{The operational flow of Algorithm 2\label{a2}}
    \begin{tabular}{p{0.93\columnwidth}}
        \hline
        \textbf{Algorithm 2:} \color{black}\textbf{Distributed NE seeking under semi-Markov switching and event-triggered mechanism}\\
        \hline
        \textbf{Initialization:}\\
        \quad For $i\in \mathscr{V}$, initial states $x_i(0)$, $v_i(0)$, $y^i(0)$.\\
        \quad Set parameters: $0<k_{1}<\frac{\mu}{\max_{i\in\mathscr{V}}\{l_i\}^2}$, $k_{i 2}>0$, \\\quad\quad\quad\quad\quad~~\quad\quad$k_{i3}>\|\dot{\omega}_i\|_{\infty}$, $k_4(r(t))$, $\zeta_i$.\\
        \quad Define sampling period $h \!>\! 0$. Initial index $q \!=\! 0$.\\
        \quad Initial trigger instant $t_0^i = 0$ and trigger counter $k = 0$.\\
        \textbf{Main Loop (Algorithm Execution):}\\
        \quad \textbf{while} $t \ge 0$ \textbf{do}\\
        \quad\quad \textbf{1. Periodic Sampling (At instant $t = qh$):}\\
        \quad\quad\quad Sample state $x_i(qh)$.\\
        \quad\quad\quad Identify the network topology mode $r(qh)$.\\
        \quad\quad \textbf{2. Event-Triggering Evaluation:}\\
        \quad\quad\quad Calculate error: $e_i(qh) = y^i(t_k^i) - y^i(qh)$.\\
        \quad\quad\quad \textbf{if} $e_i(qh)^{\rm T} \Phi e_i(qh) > \zeta_i z_i(qh)^{\rm T} \Phi z_i(qh)$ \textbf{then}\\
        \quad\quad\quad\quad \textbf{Trigger:} Set new trigger instant $t_{k+1}^i = qh$.\\
        \quad\quad\quad\quad \textbf{Broadcast:} Transmit updated estimate $y^i(t_{k+1}^i)$\\
        \quad\quad\quad\quad\quad\quad\quad~~~\ \ to neighbors under topology $r(qh)$.\\
        \quad\quad\quad\quad Update counter: $k \leftarrow k+1$.\\
        \quad\quad\quad \textbf{else}\\
        \quad\quad\quad\quad \textbf{Hold:} Maintain previous broadcast state $y^i(t_k^i)$.\\
        \quad\quad\quad \textbf{end if}\\
        \quad\quad \textbf{3. State Measurement ($t \in [qh, (q\!+\!1)h)$):}\\
        \quad\quad\quad Measure local state $x_i(t)$ and velocity $v_i(t)$.\\
        \quad\quad\quad Compute $u_i^0(t) = -k_1 \nabla_i f_i(y^i(t)) - v_i(t)$.\\
        \quad\quad \textbf{4. Compensator Variable Update:}\\
        \quad\quad\quad $s_i(t) = v_i(t) - v_i(0) - \int_0^t u_i^0(\rho)\text{d}\rho$.\\
        \quad\quad\quad $\dot{\phi}_i(t) = -k_{i3}\operatorname{sgn}(s_i(t))$.\\
        \quad\quad\quad $u_{i}^{r}(t)=-k_{i 2}\operatorname{sig}^{\frac{1}{2}}\left(s_{i}(t)\right)+\phi_{i}(t)$.\\
        \quad\quad \textbf{5. Estimation \& Control Law Computation:}\\
        \quad\quad\quad $u_i(t) = u_i^r(t)+u_i^0(t)$.\\
        \quad\quad\quad Update action estimator:\\
        \quad\quad\quad $\varepsilon\dot{y}_{j}^i(t) = -k_4(r(t))\big(\sum_{m=1}^{N} a_{i m}(r(t))(y^{i}_j(t_k^i) $\\
        \quad\quad\quad\quad\quad\quad\quad $-y^{m}_j(t_k^m))+a_{i j}(r(t))(y^{i}_j(t_k^i)-x_{j}(qh))\big)$.\\
        \quad\quad \textbf{6. State Evolution:}\\
        \quad\quad\quad Apply control input $u_i(t)$ to the physical plant:\\
        \quad\quad\quad $\dot{x}_i(t) = v_i(t)$, \ \ $\dot{v}_i(t) = \omega_i(t) + u_i(t)$.\\
        \quad\quad \textbf{7. Time Increment:}\\
        \quad\quad\quad Wait until next sampling period, then $q \leftarrow q + 1$.\\
        \quad \textbf{end while}\\
        \hline
    \end{tabular}
    \vspace{-0.6cm}
\end{table}

In Algorithm 2, $t\in[t_k^i,t_{k+1}^i)\cap[qh,(q+1)h)$, $q\in\mathbb{Z}_+$, $r(t)$ represents a continuous-time discrete-state semi-Markov process with values in a finite set $\mathbb{S}=\{1,2,\ldots,s\}$, $k_{4}(r(t))$ signifies the control gain, $a_{i m}(r(t))$ denotes the $(i,m)$-th entry of the adjacency matrix $\mathscr{A}(r(t))$.

For the sake of theoretical analysis, in this paper, the detection and implementation of event triggering only occur at a series of sampling instants. That is to say, each player's information is periodically sampled with the sampling time sequence $\{0,h,2h,\ldots\}$, where $h>0$ denotes the sampling period. Moreover, each player transfers its information to neighbors only when the event-triggered condition \eqref{chufa} is violated
\begin{equation}\label{chufa}
\begin{aligned}
t_{k+1}^{i}=&\min _{l \geq 1}\big\{t_{k}^{i}+lh \mid e _ { i } ( t _ { k } ^ { i } + l h  ) ^ { \rm T } \Phi e _ { i }(t_{k}^{i}+lh)\\
&>\zeta_{i}z_{i}\left(t_{k}^{i}+lh\right)^{\rm T}\Phi z_{i}(t_{k}^{i}+lh)\big\},
\end{aligned}
\end{equation}
where $\Phi$ is a positive definite matrix to be designed later, $z_{i j}\left(t_{k}^{i}+lh\right)=\sum_{m=1}^{N} a_{i m}(r(t_{k}^{i}+lh))\left(y^{i}_j\left(t_{k}^{i}\right)-y^{m}_j\left(t_{k}^{m}\right)\right) $ $+ a_{i j}(r(t_{k}^{i}+lh))\left(y^{i}_j\left(t_{k}^{i}\right)-x_{j}\left(t_{k}^{i}+lh\right)\right)$, $e_{i j}(t _ { k } ^ { i } + l h)=y^{i}_j\left(t_{k}^{i}\right)-y^{i}_j(t _ { k } ^ { i } + l h)$, $\delta_{ij}(t _ { k } ^ { i } + l h)=y^{i}_j(t _ { k } ^ { i } + l h)-x_{j}\left(t_{k}^{i}+lh\right)$, $\zeta_{i}>0$, $z_{i}:=\col\left(z_{i 1}, z_{i 2}, \ldots, z_{i N}\right) $, $e_{i}:=\col\left(e_{i1}, e_{i 2}, \ldots, e_{i N}\right)$, $\delta_{i}:=\col\left(\delta_{i1}, \delta_{i 2}, \ldots, \delta_{i N}\right)$.
\begin{rem}
As an extreme situation during the triggering process, Zeno behavior, an infinite number of triggers arises in a finite-time interval, may occur and cause system error if it is not strictly prohibited. To avoid the occurrence of this concern, the sampled-data-based event-triggered mechanism is introduced in this paper. Based on the triggering condition \eqref{chufa}, it can be obtained that $t_{k+1}^i-t_k^i\geq h>0$. Thus, it means that there is a fixed sampling interval greater than 0 contained in any two successive triggering instants, which strictly avoids the occurrence of the Zeno behavior.
\end{rem}

The evolution of the semi-Markov process $r(t)$ is subject to the following transition probability
\begin{small}\begin{equation*}
\operatorname{Pr}\{r(t+\vartheta)=n \mid r(t)=m\}= \begin{cases}\iota_{m n}(\vartheta) \vartheta+o(\vartheta) & m \neq n,\\ 1+\iota_{m n}(\vartheta) \vartheta+o(\vartheta) & m=n,\end{cases}\end{equation*}\end{small}where $\vartheta$ represents the sojourn time, indicating the time interval between the two consecutive successful jumps; $\iota_{m n}(\vartheta)\geq0$ signifies the transition rate from mode $m$ to mode $n$, where $n\neq m$ and $\iota_{m m}(\vartheta)=-\sum_{n=1,n\neq m}^s\iota_{m n}(\vartheta)$; $o(\vartheta)$ is the higher order infinitesimal operator given as $\lim_{\vartheta\rightarrow0}\frac{o(\vartheta)}{\vartheta}=0$.
\begin{rem}
Compared to the NE seeking algorithm under standard Markov switching topologies in \cite{fang2019distributed}, Algorithm 2 overcomes three major challenges: 1) Communication topology: The semi-Markov switching topologies alleviate the memoryless restriction possessed by Markov chains with exponential distributions, allowing for a more general probability distribution of sojourn time (e.g. Weibull). However, the time-varying transition rates invalidate the static algebraic Lyapunov tools utilized in \cite{fang2019distributed}. To overcome this, a novel piecewise stochastic Lyapunov-Krasovskii functional is constructed. 2) Energy efficiency: Unlike the constant communication in \cite{fang2019distributed}, Algorithm 2 incorporates a sampled-data-based event-triggered mechanism to reduce communication burden. The introduction of this mechanism results in a delay differential equation structure of the algorithm, necessitating advanced input-delay modeling and Jensen's inequality to establish the algorithm's mean-square stability. 3) Algorithm execution: Algorithm 2 operates in a continuous-time setting, which avoids the complicated step size design required by the discrete-time algorithm in \cite{fang2019distributed}, thereby decreasing the complexity of parameter tuning. To establish the stability analysis, the weak infinitesimal operator and Dynkin's formula are employed.
\end{rem}
\begin{rem}
Note that in Algorithm 2, we consider the second-order dynamics in \eqref{erjie} in the absence of uncertain dynamics $\varrho_i(\mathbf{x})$, because the following average tracking protocol with periodic sampling and event-triggered mechanism is invalid
\begin{equation*}
\begin{aligned}
 \eta_i&=\beta_i+\|v_i\|,\\
\dot{\beta}_i&=-\alpha\sum_{j=1}^Na_{ij}(r(t))\sign(\eta_i(t_k^i)-\eta_j(t_k^j)).
\end{aligned}
\end{equation*}
Due to the influence of periodic sampling, when the consensus error is small, the above protocol may experience severe fluctuations and cannot achieve average tracking. The strict proof is explicitly given in Section IV-B of \cite{chen2012distributed}.
\end{rem}

To simplify the analysis in the subsequent section, the following assumption is provided.
\begin{assum}\label{jiashe31}
Each probable undirected graph $\tilde{G}_i$,$~i\in\mathbb{S}$, is connected.
\end{assum}
\begin{rem}\label{zhujie33}
Assumption \ref{jiashe31} guarantees the feasibility of the linear matrix inequality \eqref{juzhen1} in Theorem \ref{dinglijunfang}, which has been widely adopted in existing research on semi-Markov (or Markov) switching topologies \cite{zhang2024time,dai2018event,fang2019distributed}. In this paper, this assumption is adopted to maintain a focused presentation on the joint design of semi-Markov switching and event-triggered mechanisms. However, it is worth noting that this condition can actually be relaxed. By exploiting the ergodic property of the semi-Markov chain, the proposed stability results can be naturally extended to accommodate intermittently disconnected topologies, provided that the overall network remains jointly connected in the probabilistic sense.
\end{rem}
\begin{assum}\label{jiashe32}
The set of all possible topologies $\mathbb{S}$ can be partitioned into two subsets: $\mathbb{S}_c$ (the set of connected graphs) and $\mathbb{S}_d$ (the set of disconnected graphs). The semi-Markov process $r(t)$ is ergodic with a unique stationary probability distribution $\pi = [\pi_1, \pi_2, \ldots, \pi_s]$. The overall network is jointly connected in the probabilistic sense.
\end{assum}

Similar to Algorithm 1, we can decompose Algorithm 2 into three parts the same way.
  \begin{itemize}
  \item Second-order NE seeking
  \begin{equation*}
  \left\{\begin{aligned}
  \dot{x}_{i}&=v_{i}, \\
\dot{v}_{i}&=\omega_{i}+u_{i},\\
u_i&=u_i^0+u_i^r,\\
u_{i}^{0}&=-k_{1} \nabla_{i} f_{i}\left(y^{i}\right)-v_i.
\end{aligned}\right.
\end{equation*}
  \item Finite-time compensating mechanism for disturbances
  \begin{equation*}
  \left\{\begin{aligned}
  u_{i}^{r}&=-k_{i 2} \operatorname{sig}^{\frac{1}{2}}\left(s_{i}\right)+\phi_{i},\\
\dot{\phi}_{i}&=-k_{i3} \operatorname{sgn}\left(s_{i}\right),\\
s_{i}&=v_{i}-v_{i}(0)-\int_{0}^{t} u_{i}^{0}(\rho) d \rho.\\
\end{aligned}\right.
\end{equation*}
  \item Action estimation
    \begin{equation}\label{guji2}
  \begin{aligned}
  \varepsilon\dot{y}^{i}_j=&-k_{ 4}(r(t))\left(\sum_{m=1}^{N} a_{i m}(r(t))\left(y^{i}_j\left(t_{k}^{i}\right)-y^{m}_j\left(t_{k}^{m}\right)\right)\right.\\
  &+a_{i j}(r(t))\left(y^{i}_j\left(t_{k}^{i}\right)-x_{j}\left(qh\right)\right)\Big).
\end{aligned}
\end{equation}
\end{itemize}

Since Algorithm 2 has a two-time-scale structure, singular perturbation methods are used to separately analyze different time scales. The boundary-layer system is given as
\begin{small}\begin{equation}\label{bianjie2}
  \begin{aligned}
\frac{\text{d}{y}_{j}^i}{\text{d}\tau}=&-k_{ 4}(r(\varepsilon\tau))\left(\sum_{m=1}^{N} a_{i m}(r(\varepsilon\tau))\left(y^{i}_j\left(\varepsilon\tau_{k}^{i}\right)-y^{m}_j\left(\varepsilon\tau_{k}^{m}\right)\right)\right.\\
  &+a_{i j}(r(\varepsilon\tau))\left(y^{i}_j\left(\varepsilon\tau_{k}^{i}\right)-x_{j}\left(qh\right)\right)\Big),
\end{aligned}
\end{equation}\end{small}where $\tau\in[\tau_k^i,\tau_{k+1}^i)\cap[qh/\varepsilon,(q+1)h/\varepsilon)$, $q\in\mathbb{Z}_+$. In the boundary-layer system \eqref{bianjie2}, $x_j(qh)$ can be viewed as a fixed vector. Substituting the quasi-steady state $\mathbf{y}_{\sigma}:=\col(y^1_{\sigma},y^2_{\sigma},\ldots,y^N_{\sigma})=\mathbf{1}_N\otimes\mathbf{x}$ into \eqref{bianjie2}, we can get the reduced system \begin{equation*}\label{jianhua2}
  \left\{\begin{aligned}
  \dot{x}_{i}&=v_{i}, \\
\dot{v}_{i}&=\omega_{i}+u_i^r-k_{1} \nabla_{i} f_{i}\left(\mathbf{x}\right)-v_i,\\
u_i&=u_i^0+u_i^r,\\
\dot{\phi}_{i}&=-k_{i3} \operatorname{sgn}\left(s_{i}\right).\\
\end{aligned}\right.
\end{equation*}
Note that the convergence of the above reduced system have been theoretically analyzed in Section \ref{sec3} with the absence of uncertain dynamics $\varrho(\mathbf{x})$. Thus, the rest of the discussion mainly concentrates on the boundary-layer system \eqref{bianjie2}.
\begin{rem}
Compared to the boundary-layer system \eqref{bianjie1}, the analysis of the boundary-layer system \eqref{bianjie2} presents two main challenges: 1) The introduction of the semi-Markov process $r(t)$ results in the variations of $y^i_j$ becoming unmeasurable, where the following mean-square consensus in terms of expectation should be considered. 2) The implementation of a sampled-data-based event-triggered mechanism aims to significantly reduce unnecessary communication among players, leading to the discontinuity of the boundary-layer system \eqref{bianjie2}. However, through variable transformation, it can be reformulated as a delay differential equation. Consequently, the subsequent analysis focuses on the proof of the mean-square consensus of the state of this delay differential equation.
\end{rem}
\begin{defn}
Under semi-Markov switching topologies, the mean-square consensus of the state of the boundary-layer system \eqref{bianjie2} is said to be achieved if
\begin{equation*}
\lim_{\tau\rightarrow\infty}\mathbb{E}\|y^i(\tau)-\mathbf{x}\|^2=0,~\forall i\in\mathscr{V},
\end{equation*}
holds for any initial distribution $r_0\in\mathbb{S}$ and any initial condition $\psi(\pi)$, $\forall\pi\in[-\frac{h}{\varepsilon},0]$.
\end{defn}

To facilitate the analysis on the boundary-layer system \eqref{bianjie2}, substituting $e_{ij}$ and $\delta_{ij}$ into \eqref{bianjie2}, we derive
\begin{small}
\begin{equation*}
\begin{aligned}
\frac{\text{d}{y}_{j}^i}{\text{d}\tau}=&-k_{ 4}(r(\varepsilon\tau))\left[\sum _ { m = 1 } ^ { N } a _ { i m } ( r (\varepsilon\tau) ) \left[e_{i j}(q h)-e_{m j}(q h)\right.\right.\\
&+\delta_{i j}(q h)\left.-\delta_{m j}(q h)\right]+a_{i j}\left(r(\varepsilon\tau)\right)\left(e_{i j}(q h)+\delta_{i j}(q h)\right)\bigg] \\
=&-k_4(r(\varepsilon\tau))\left(\sum_{m=1}^{N} l _{i m}(r(\varepsilon\tau))\left(e_{i j}(q h)+\delta_{i j}(q h)\right)\right.\\
&+a_{i j}(r(\varepsilon\tau))\left(e_{i j}(q h)+\delta_{i j}(q h)\right)\Bigg),
\end{aligned}
\end{equation*}
\end{small}where $l _{i m}(r(\varepsilon\tau))$ denotes the $(i,m)$-th entry of the Laplacian matrix $L(r(\varepsilon\tau))$. Note that, in the $\tau$ time scale, $\mathbf{x}$ can be treated as a fixed vector. Thus, for $r(\varepsilon\tau)=m,~m\in\mathbb{S}$, we can get
\begin{equation}
\begin{aligned}
\frac{\text{d}{\delta}}{\text{d}\tau}=&-\left[\left(L(m)\otimes \mathbf{I}_N+A_0(m)\right)\otimes\mathbf{I}_n \otimes K(m)\right]\\
&\times[e(\tau-\varsigma(\tau))+\delta(\tau-\varsigma(\tau))],\label{dde1}
\end{aligned}
\end{equation}where $\delta=\col(\delta_1,\delta_2,\ldots,\delta_N)$, $e=\col(e_1,e_2,\ldots,e_N)$, $A_0(m)=\diag_{i,j\in\mathcal{V}}\{a_{ij}(m)\}$, $K(m)=k_4(m)$, $\tau-\varsigma(\tau)=\frac{qh}{\varepsilon}$. For ease of discussion, we omit the $\varepsilon$ in front of $\tau$. It is easily obtained that $\varsigma(\tau)\in[0,\frac{h}{\varepsilon})$, which is a piecewise-linear function with $\dot{\varsigma}(\tau)=1$ for $\tau\neq \frac{qh}{\varepsilon}$. Then, for the delay differential equation \eqref{dde1}, the initial condition for $\delta(\tau)$ is supplemented as $\delta(\pi)=\psi(\pi)$, $\pi\in[-\frac{h}{\varepsilon},0]$ with $\psi(\pi)\in\mathcal{C}([-\frac{h}{\varepsilon},0], \mathbb{R}^{N^2n}),$ where $\mathcal{C}([-\frac{h}{\varepsilon},0], \mathbb{R}^{N^2n})$ represents the continuous functions mapping from $[-\frac{h}{\varepsilon},0]$ into $\mathbb{R}^{N^2n}$ with the norm $\|\psi\|_{\mathcal{C}}=\sup_{c\in[-\frac{h}{\varepsilon},0]}\|\psi(c)\|.$

Next, we can conclude the discussion on boundary-layer system \eqref{bianjie2} in the following theorem.
\begin{thm}\label{dinglili3}
Under Assumption \ref{jiashe31}, the state of the boundary-layer system \eqref{bianjie2} with event-triggered mechanism \eqref{chufa} is mean-square stable, i.e, the state of the system \eqref{bianjie2} achieves consensus in the sense of mean-square, if there exist positive definite matrices $\mathds{Q}$, $\mathds{U}$, $\mathds{R}$ and $\mathds{P}(m)$, and matrix $\mathds{S}$, for $m\in\mathbb{S}$, such that the following matrix inequalities hold:
\begin{align}
&\left[\begin{array}{cc}
\Xi_1(m) & \frac{h}{\varepsilon} \mathds{B}^{\rm T}(m) \\
* & -\left(\mathbf{I}_{Nn} \otimes \mathds{R}\right)^{-1}
\end{array}\right]\prec0, ~m\in\mathbb{S},\label{juzhen1}\\
&\left[\begin{array}{cc}
\mathbf{I}_{Nn} \otimes \mathds{R} & \mathds{S} \\
\mathds{S}^{\rm T} & \mathbf{I}_{Nn} \otimes \mathds{R}
\end{array}\right]\succ0,\label{juzhen2}
\end{align}
where $\mathcal{H}(m)=(L(m)\otimes \mathbf{I}_N+A_0(m))\otimes\mathbf{I}_{n}$, $\Lambda=\diag_{i\in\mathcal{V}}\{\zeta_i\}$, $\text{\uppercase\expandafter{\romannumeral2}}_i = e_i^{\rm T} \otimes \mathbf{I}_{N^2n}$ $(i=1,2,3,4)$ with $e_i \in \mathbb{R}^4$ being the $i$-th vector of the canonical basis, $\text{\uppercase\expandafter{\romannumeral2}}_{jk} = \text{\uppercase\expandafter{\romannumeral2}}_j - \text{\uppercase\expandafter{\romannumeral2}}_k$, $\mathbb{B}(m)=-(\mathcal{H}(m) \otimes K(m))( \text{\uppercase\expandafter{\romannumeral2}}_2+\text{\uppercase\expandafter{\romannumeral2}}_4)$, and \begin{small}\begin{align*}
\Xi_1(m)=&\sum_{a=1}^s \lambda_{m n}(\vartheta) \text{\uppercase\expandafter{\romannumeral2}}_1^{\rm T}\left(\mathbf{I}_{Nn} \otimes \mathds{P}(a)\right) \text{\uppercase\expandafter{\romannumeral2}}_1-\text{\uppercase\expandafter{\romannumeral2}}_3^{\rm T}\left(\mathbf{I}_{Nn} \otimes \mathds{Q}\right) \text{\uppercase\expandafter{\romannumeral2}}_3\\
&+\text{\uppercase\expandafter{\romannumeral2}}_1^{\rm T}\left[\mathbf{I}_{Nn} \otimes(\mathds{Q+U})\right] \text{\uppercase\expandafter{\romannumeral2}}_1-\text{\uppercase\expandafter{\romannumeral2}}_4^{\rm T}\left(\mathbf{I}_{N}\otimes \Phi\right) \text{\uppercase\expandafter{\romannumeral2}}_4\\
& +\text{\uppercase\expandafter{\romannumeral2}}_1^{\rm T}\left(\mathbf{I}_{Nn} \otimes \mathds{P}(m)\right) \mathds{B}(m)+\mathds{B}(m)^{\rm T}\left(\mathbf{I}_{Nn} \otimes \mathds{P}(m)\right) \text{\uppercase\expandafter{\romannumeral2}}_1\\
&-\left[\begin{array}{c}
\text{\uppercase\expandafter{\romannumeral2}}_{23} \\
\text{\uppercase\expandafter{\romannumeral2}}_{12}
\end{array}\right]^{\rm T}\left[\begin{array}{cc}
\mathbf{I}_{Nn} \otimes \mathds{R} & \mathds{S} \\
\mathds{S}^{\rm T} & \mathbf{I}_{Nn} \otimes \mathds{R}
\end{array}\right]\left[\begin{array}{c}
\text{\uppercase\expandafter{\romannumeral2}}_{23} \\
\text{\uppercase\expandafter{\romannumeral2}}_{12}
\end{array}\right]\\
&+\left(\text{\uppercase\expandafter{\romannumeral2}}_2+\text{\uppercase\expandafter{\romannumeral2}}_4\right)^{\rm T}\left(\mathcal{H}(m)^{\rm T} (\Lambda \otimes \Phi) \mathcal{H}(m)\right)\left(\text{\uppercase\expandafter{\romannumeral2}}_2+\text{\uppercase\expandafter{\romannumeral2}}_4\right).
\end{align*}\end{small}
\end{thm}
\begin{proof} See Appendix \ref{fulu1}. \end{proof}
\begin{rem}
Based on Theorem \ref{dinglili3}, the position estimations on non-neighboring players converge to the actual positions. Compared with the results in \cite[Proposition 1]{zhang2017improved}, fewer free slack matrix variables are introduced, which reduces the computational costs of matrix inequalities to a great extent.
\end{rem}
\begin{thm}\label{dinglijunfang}
Suppose that Assumptions \ref{jiashe22}, \ref{jiashe25}, \ref{jiashe31} and the matrix inequalities \eqref{juzhen1} and \eqref{juzhen2} hold. For Algorithm 2, when $t\geq T_{\max}$, if $0<k_1<\frac{\mu}{\max_{i\in\mathscr{V}}\{l_i\}^2}$ and $k_{i3}>\|\dot{\omega}_i\|_{\infty}$ for $i\in\mathscr{V}$, there exists $\varepsilon^*>0$ such that for each $\varepsilon \in\left(0, \varepsilon^*\right]$,  the measurement errors $\mathfrak{Q}(t)$ and $\delta(t)$ are mean-square stable.
\end{thm}
\begin{proof} According to Lemma \ref{yinlili1}, when $t\geq T_{\max}$, Algorithm 2 can be divided as the boundary-layer system \eqref{bianjie2} and the reduced system \eqref{000}. Then, based on Theorems \ref{dinglili1} and \ref{dinglili3}, the state of the boundary-layer system \eqref{bianjie2} is mean-square stable and the state of the reduced system \eqref{000} is exponentially stable. Based on the Lyapunov function settings in Theorems \ref{dinglili1} and \ref{dinglili3}, assumptions (23)-(27) and conditions (33)-(34) in \cite{seroka2013mean} can be verified. Then, the proof of Theorem \ref{dinglijunfang} follows a similar argument as in Corollary 2 of \cite{seroka2013mean}.\end{proof}
\begin{rem}
Currently, the effectiveness of Theorems \ref{dingliyuan} and \ref{dinglijunfang} relies strictly on the strong monotonicity of the pseudo-gradient and the convexity of the players' cost functions. These conditions are essential for ensuring the exponential convergence of the reduced system, which in turn guarantees the effectiveness of the singular perturbation technique utilized in this paper. Nevertheless, functional non-convexity is ubiquitous in practice, making the extension to non-convex games a challenging yet highly practical future direction. By integrating mirror descent methods \cite{chen2024approaching}, the proposed algorithm could potentially be extended to seek a global NE for non-convex games in a canonical form. Additionally, leveraging collaborative neurodynamic approaches \cite{xia2023collaborative} may ensure players to escape local NEs and converge to the global NE in more general non-convex games. Addressing these non-convex games while maintaining the inherent robustness of the algorithms under complex environments constitutes an important part of our future research.
\end{rem}
\begin{corollary}\label{tuilun111}
Under Assumptions \ref{jiashe22}, \ref{jiashe25}, \ref{jiashe32} and the matrix inequality \eqref{juzhen2}, suppose there exist scalars $\imath_m > 0$ (for $m \in \mathbb{S}_c$) and $\jmath_m > 0$ (for $m \in \mathbb{S}_d$) such that
\begin{align}
\Psi(m) &\preceq -\imath_m \mathbf{I}, \quad \forall m \in \mathbb{S}_c, \label{eq_stable_mode} \\
\Psi(m) &\preceq \jmath_m \mathbf{I}, \quad \forall m \in \mathbb{S}_d, \label{eq_unstable_mode}
\end{align}
where $\Psi(m)=\left[\begin{array}{cc}
\Xi_1(m) & \frac{h}{\varepsilon} \mathds{B}^{\rm T}(m) \\
* & -\left(\mathbf{I}_{Nn} \otimes \mathds{R}\right)^{-1}
\end{array}\right].$ Then, for Algorithm 2, when $t\geq T_{\max}$, if $0<k_1<\frac{\mu}{\max_{i\in\mathscr{V}}\{l_i\}^2}$ and $k_{i3}>\|\dot{\omega}_i\|_{\infty}$ for $i\in\mathscr{V}$, there exists $\varepsilon^*>0$ such that for each $\varepsilon \in\left(0, \varepsilon^*\right]$,  the measurement errors $\mathfrak{Q}(t)$ and $\delta(t)$ are mean-square stable.
\end{corollary}
\begin{proof}
See Appendix \ref{fulu2}.
\end{proof}
\section{Simulation and Numerical Examples}\label{sec5}
\subsection{Connectivity Control Game Statement}
Inspired by \cite{stankovic2011distributed},\begin{figure}[!htpb]
\centerline{\includegraphics[width=0.55\linewidth]{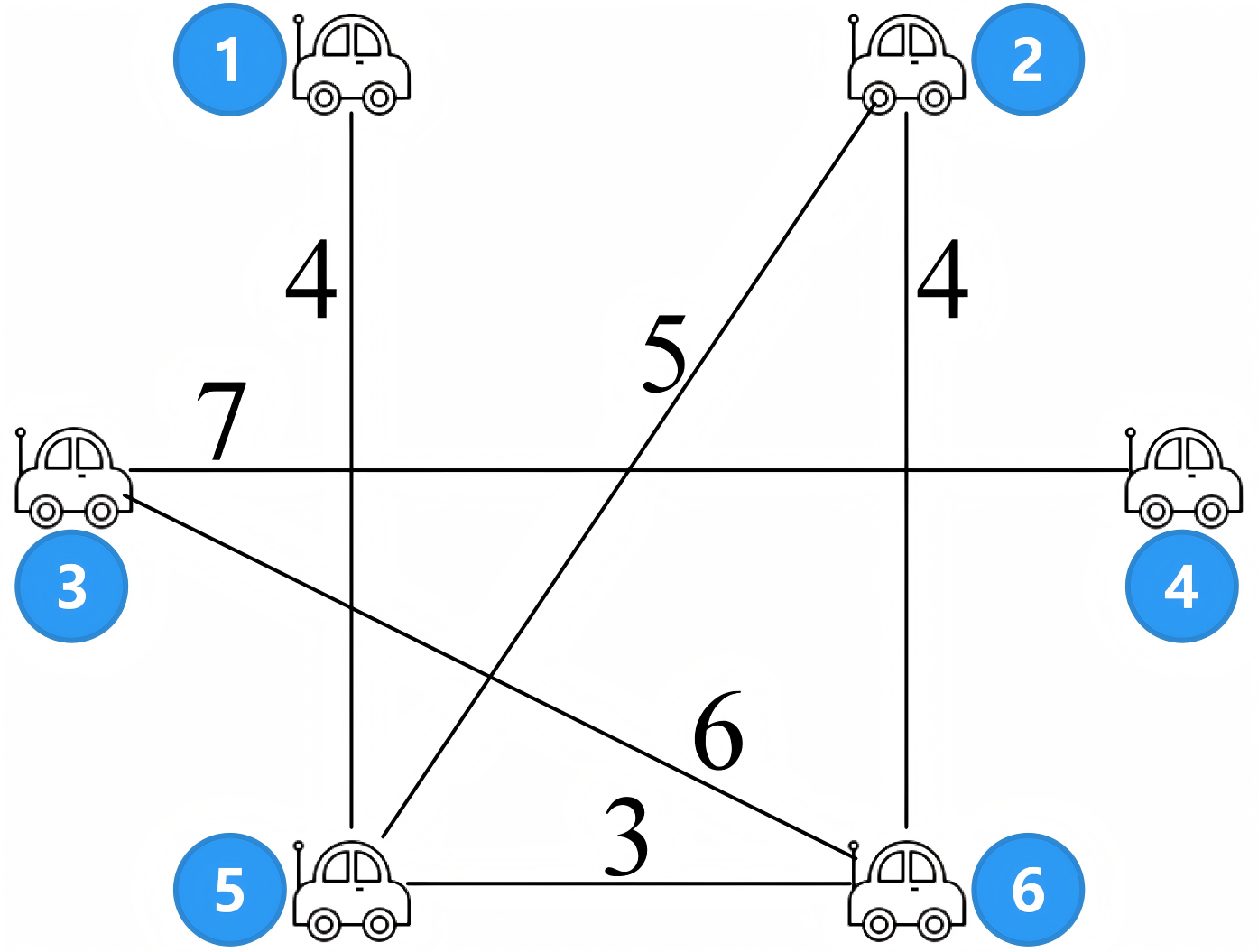}}
\caption{Communication topology of the vehicles.}\label{tutu1}
\label{f12}
\end{figure}\begin{figure}[!htbp]
\centering{
\begin{minipage}{0.8\linewidth}
\includegraphics[width=\linewidth]{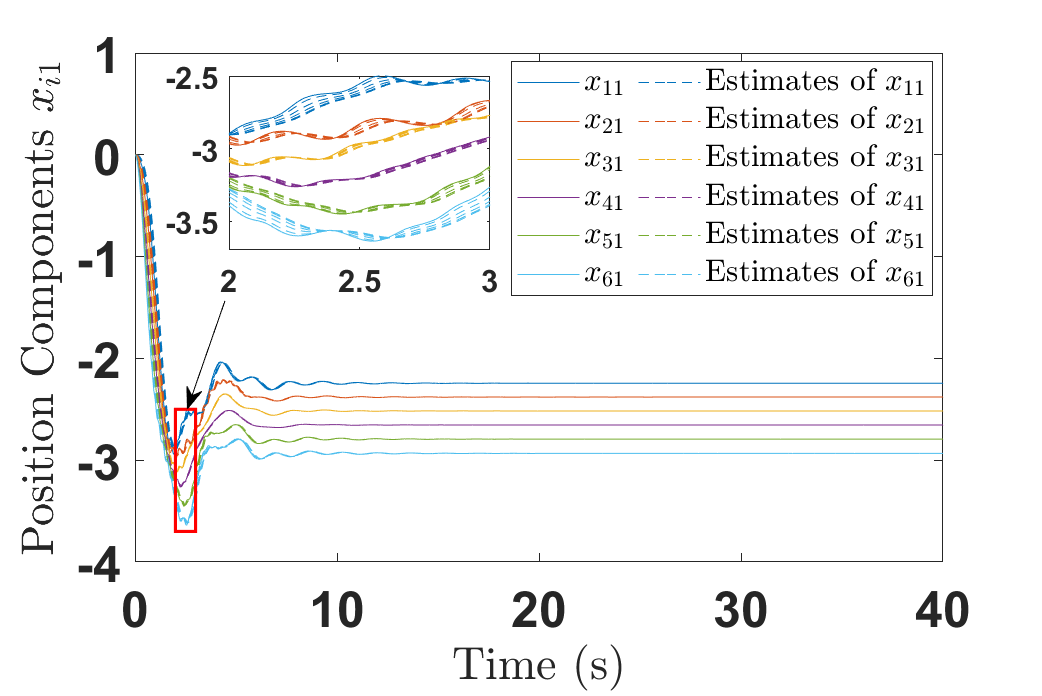}
		\label{f13}
\end{minipage}}
\centering{
\begin{minipage}{0.8\linewidth}
\includegraphics[width=\linewidth]{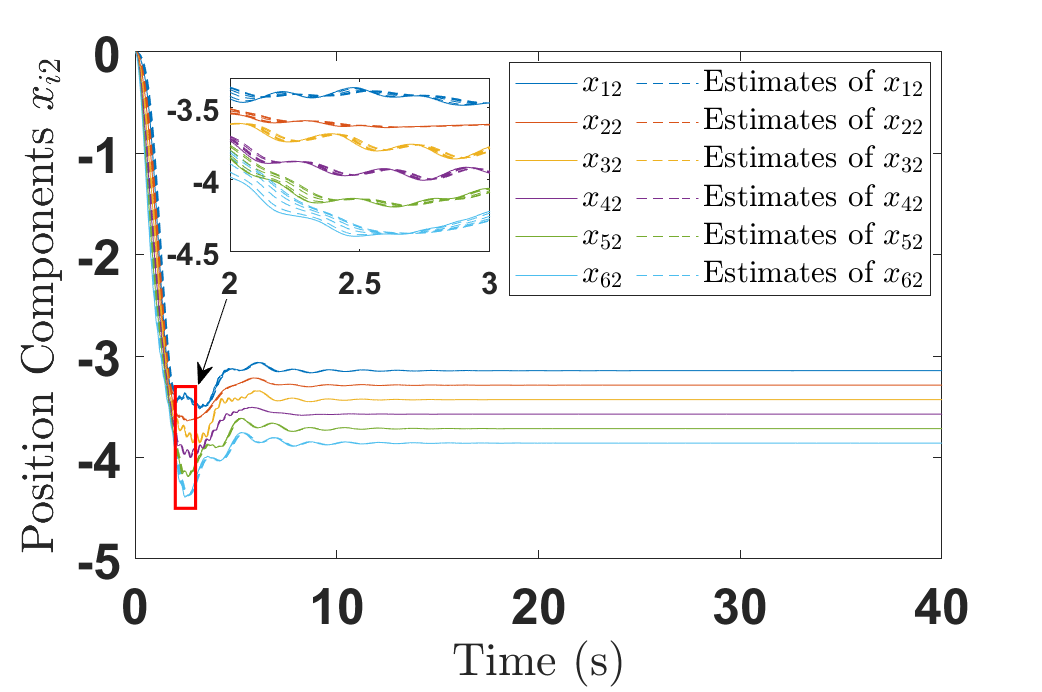}
		\label{f14}
\end{minipage}}
	\caption{$x_{i1}(t)$, $x_{i2}(t)$ for the position trajectories produced by Algorithm 1.}\label{tutu2}
	\label{ff}
\end{figure}\begin{figure}[!htpb]
\centerline{\includegraphics[width=0.8\linewidth]{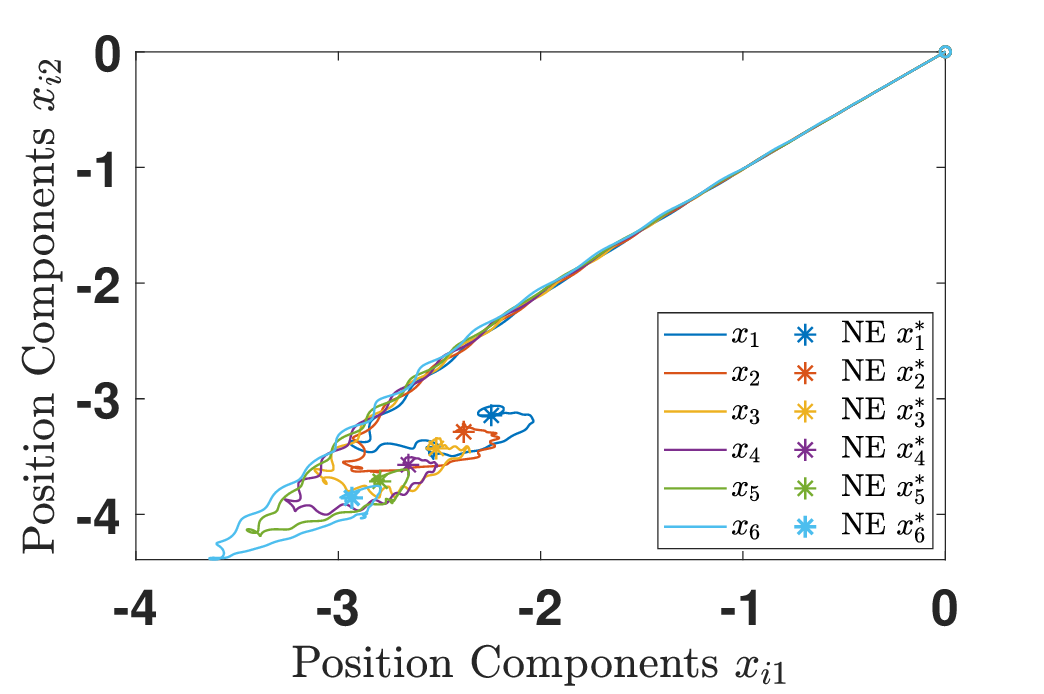}}
\caption{All the vehicles' trajectories in 2D plane.}\label{tutu3}
\label{f12}
\vspace{-0.5cm}
\end{figure} in this subsection, we formulate a connectivity control problem of mobile sensor vehicles to verify the effectiveness of Algorithms 1 and 2. Specifically, a group $\mathfrak{N}=\{1,2,\ldots,6\}$ of mobile sensor vehicles is employed to coordinate their positions through wireless communication to fulfill complex tasks such as reconnaissance or exploration. 
Due to the inherent physical inertia of mobile sensor vehicles, unlike idealized first-order kinematic systems that assume instantaneous velocity changes, actual vehicles are driven by acceleration (i.e., force or torque) commands \cite{yang1999sliding}. Consequently, modeling these vehicles as second-order dynamic systems is physically indispensable for accurately capturing their real-world mechanical behaviors and ensuring the practical realizability of the proposed controllers. Therefore, the second-order system \eqref{erjie} is utilized to model the movement of the vehicles, where $x_i:=\col(x_{i1},x_{i2})$ denotes vehicle $i$'s planar position and $v_i:=\col(v_{i1},v_{i2})$ represents its velocity.

In the simulation, each vehicle $i$ attempts to autonomously find the target location to optimize its private objective $C_i(x_i)$, while maintaining relative positions with other vehicles. Thus, vehicle $i$'s ultimate goal can be formulated as follows
\begin{equation}\label{li1}\begin{array}{l}
\mbox{min  }~~  J_i(x_i,x_{-i})=C_i(x_i)+\sum_{j\in\mathfrak{N}}\|x_i-x_j\|^2,
\end{array}\end{equation}
where $C_i(x_i)=\|x_i\|^2+c_i^{\rm T}x_i+\sin(x_{i1})$, with a randomly generated private parameter vector $c_i\in\mathbb{R}^2$. The private objective $C_i(x_i)$ represents the specific task penalty for vehicle $i$, which comprises two parts. The quadratic portion, $\|x_i\|^2+c_i^{\rm T}x_i$, evaluates the distance penalty between the vehicle's actual position and its nominal target, where $c_i$ encodes the target's spatial coordinates. Furthermore, the trigonometric term $\sin(x_{i1})$ reflects periodic terrain variations (e.g., undulating surfaces or sand dunes) encountered during the exploration.

Based on analytical calculation, we obtain $\mu=2$ and $l_i=22$ for $i\in\mathfrak{N}$. By setting $\col\{c_1,c_2,\ldots,c_6\}$ as $\col\{1,2,\ldots,12\}$, the NE of the connectivity control game \eqref{li1} is calculated as $\col\{x_1^*,x_2^*,\ldots,x_6^*\}=\col\{-2.245,-3.14,-2.38,-3.28,-2.51,-3.42,-2.65,$ $-3.56,-2.8,-3.71,-2.95,-3.85\}$.
\subsection{Experimental Results of Algorithm 1}
Influenced by wireless network signal fluctuation and signal loss, each vehicle's control input $u_i$ may inevitably encounter external disturbances and uncertain dynamics.
\begin{figure}[!htpb]
\centerline{\includegraphics[width=0.8\linewidth]{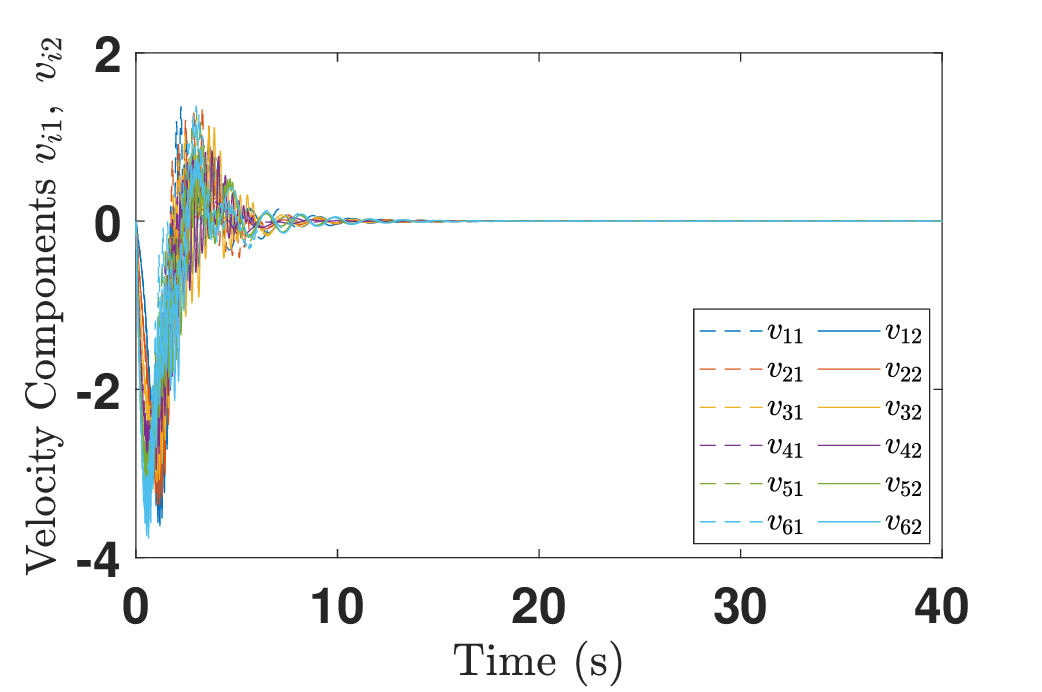}}
\caption{$v_{i1}(t)$, $v_{i2}(t)$ for the velocity trajectories produced by Algorithm 1.}\label{tutu4}
\label{f12}
\vspace{-0.5cm}
\end{figure}\begin{figure*}[!h]
\centering{
\begin{minipage}{0.3\linewidth}
\includegraphics[width=0.9\linewidth]{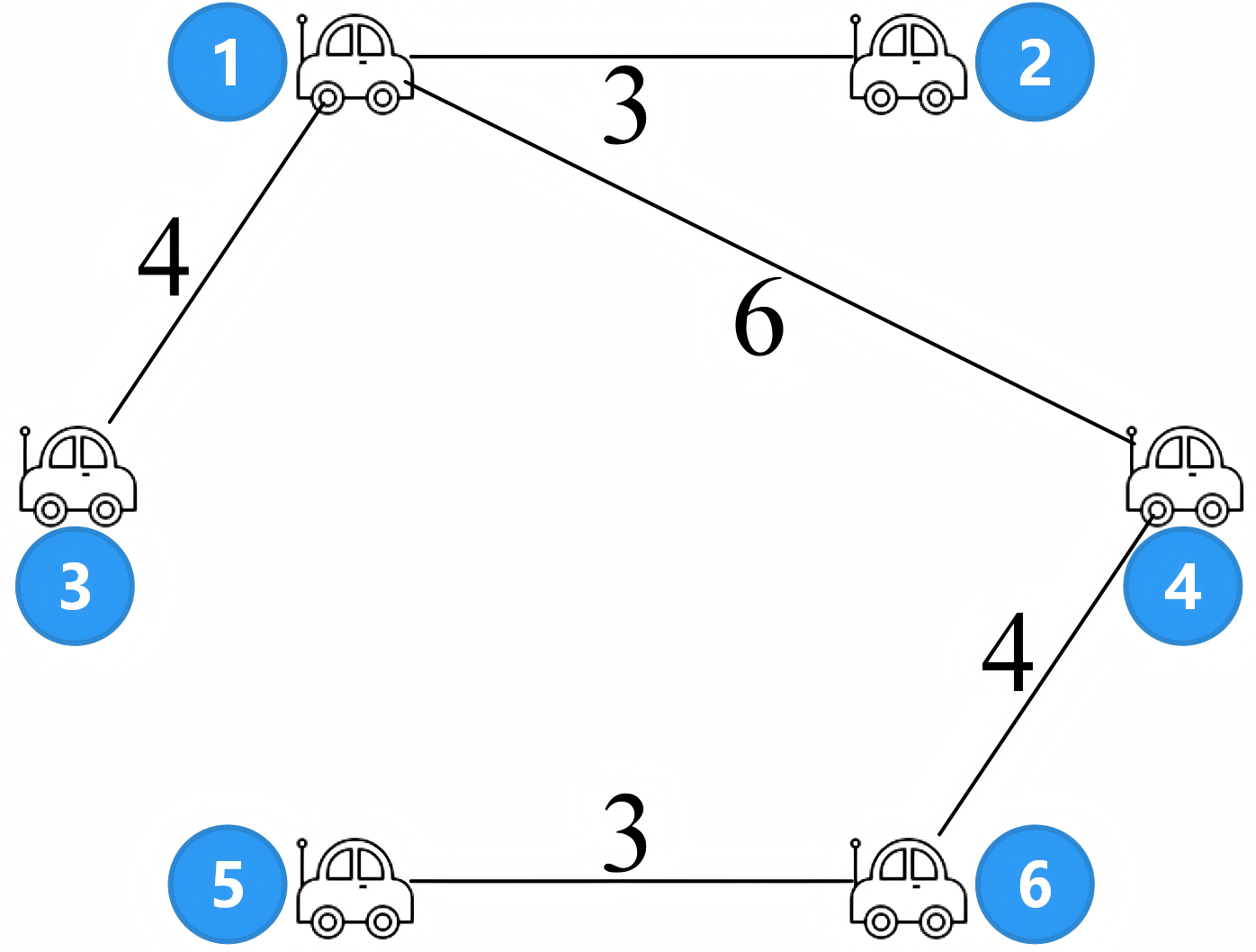}

\end{minipage}}
\centering{
\begin{minipage}{0.3\linewidth}
\includegraphics[width=0.9\linewidth]{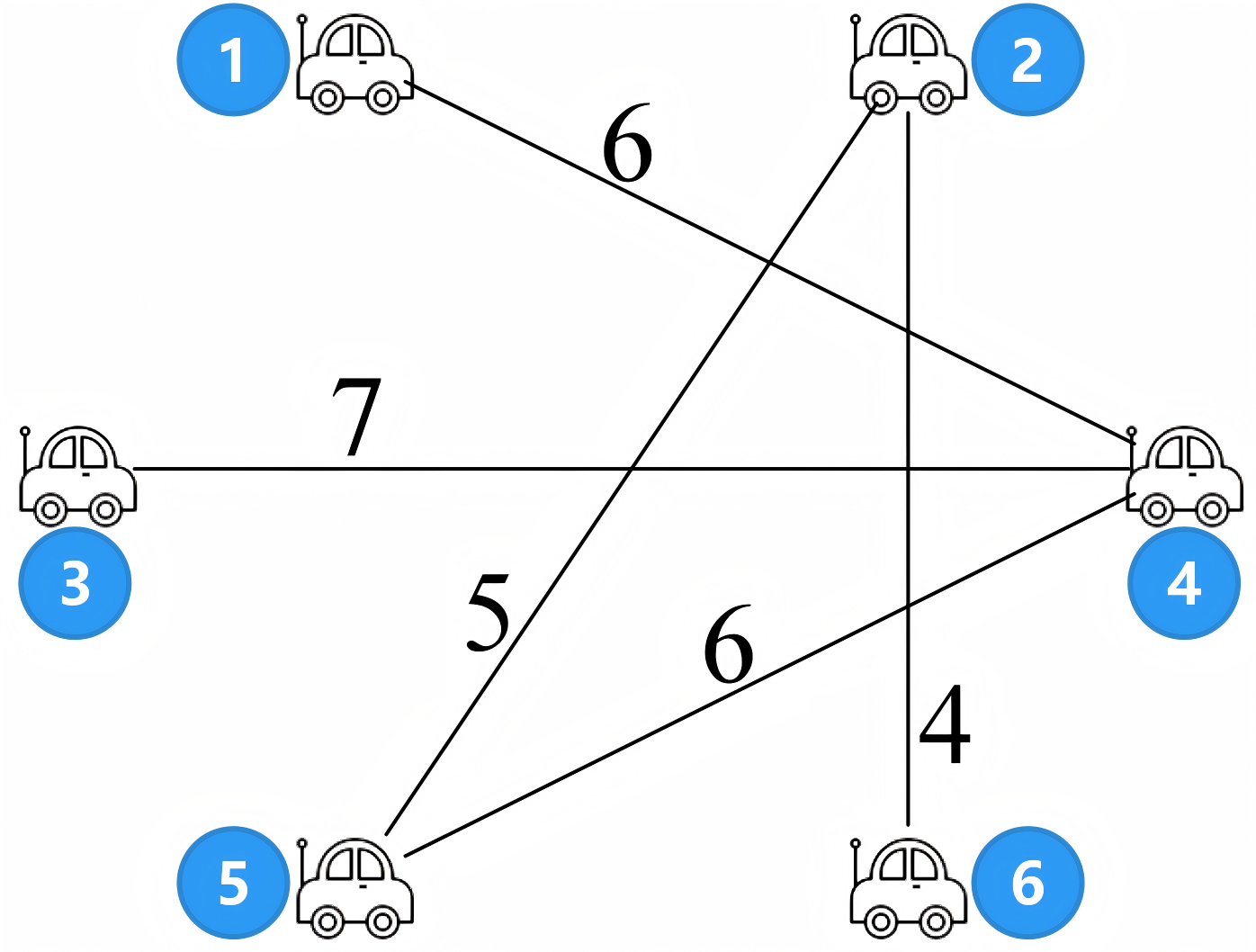}

\end{minipage}}
\centering{
\begin{minipage}{0.3\linewidth}
\includegraphics[width=0.9\linewidth]{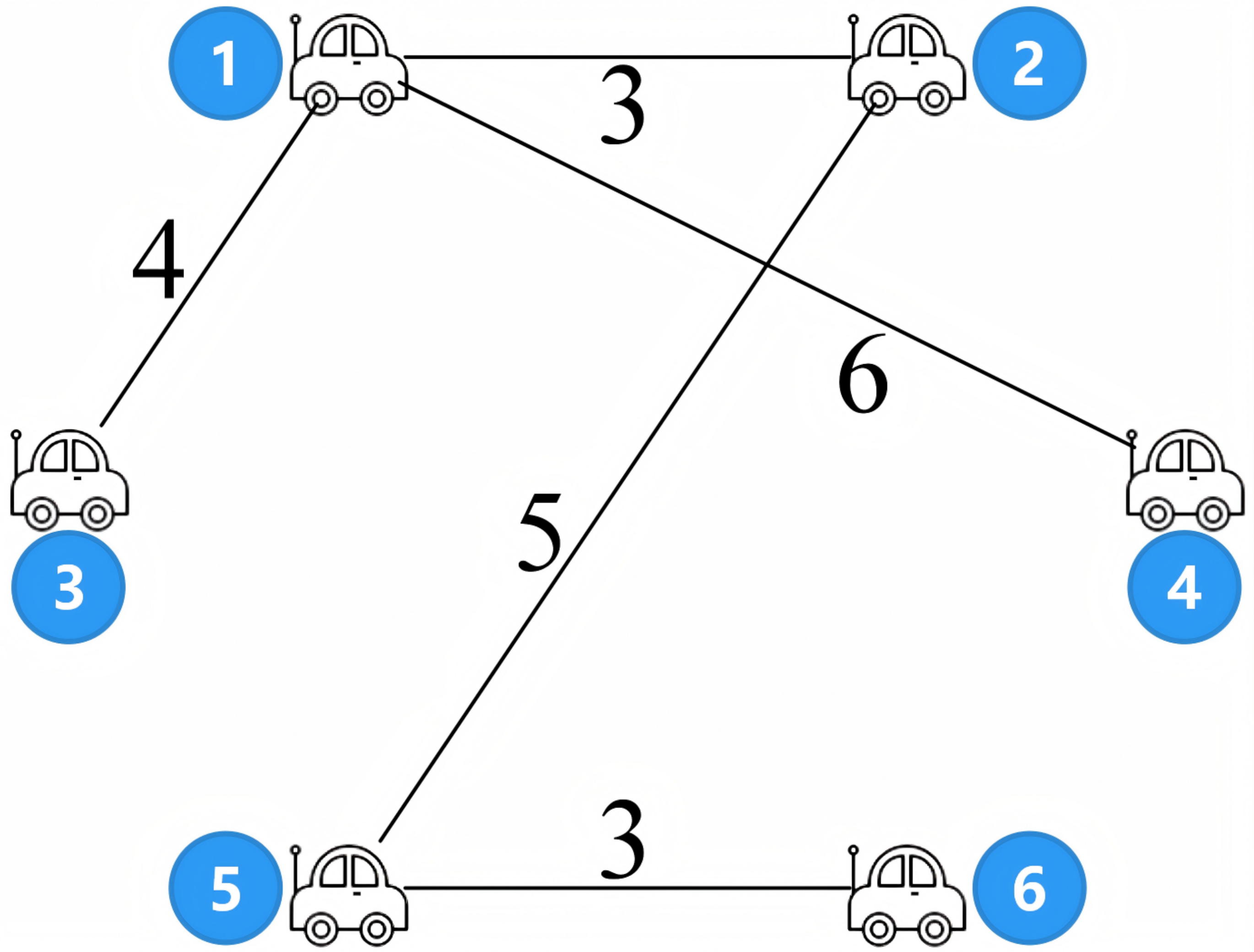}
\end{minipage}}
	\caption{Three possible switching topologies.}\label{tutu5}
\end{figure*} These external factors can degrade the performances of some NE seeking algorithms. To verify the interference rejection ability of Algorithm 1, the subsequent simulation is conducted based on the settings below.
\begin{itemize}
\item All the vehicles communicate with each other through an undirected connected graph shown in Fig. \ref{tutu1}.
  \item The disturbance for vehicle $i$ is set as $\omega_i(t):=\col(\sin(it),\cos(it))$.
  \item As discussed in \cite[Section II]{yao2015rise}, during the movement of vehicles, the uncertain friction effects can be lumped into the unmodeled term $\varrho_i(\mathbf{x})$. Due to the difficulty in quantitatively estimating the friction, in this connectivity control game, we use linear functions to represent the uncertain dynamics defined as follows: for vehicles $i\in\{1,2,\ldots,5\}$, let $\varrho_i(\mathbf{x}):=\col(ix_{i1}+ix_{i2}+(i+1)x_{i+1,1}+(i+1)x_{i+1,2},~3ix_{i1}+2ix_{i2}+3(i+1)x_{i+1,1}+2(i+1)x_{i+1,2})$, and the one for vehicle 6 is $\varrho_6(\mathbf{x}):=\col(6x_{61}+6x_{62}+x_{11}+x_{12},~18x_{61}+12x_{62}+3x_{11}+2x_{12})$.
  \item Some parameters in Algorithm 1 is designed: $k_{1}=0.001$, $k_{i3}=5$, $k_{i2}=k_{i4}=1$ for $i\in\mathfrak{N}$, $\tilde{g}=20$, $\alpha=20$, $\varepsilon=0.01$.
\end{itemize}

It is obvious that Assumptions \ref{jiashe1}-\ref{jiashe25} are satisfied. The simulation results produced by Algorithm 1 are given in Figs. \ref{tutu2}-\ref{tutu4}. 
As depicted in Fig. \ref{tutu2}, despite the initial fluctuations induced by external disturbances and uncertain dynamics, all vehicles' positions successfully converge to the NE of the noncooperative game \eqref{li1}. Simultaneously, all the other vehicles' position estimations on one vehicle converge to its actual positions. Furthermore, the transient behavior in Fig. \ref{tutu2}, specifically, the complete elimination of disturbance-induced ripples after $t \approx 15$s, provides that the proposed ISMC controller $u_i^r$ effectively rejects all disturbances in finite time, thereby corroborating the analysis in Lemma \ref{yinlili1}. To provide a more intuitive physical perspective, Fig. \ref{tutu3} illustrates the planar spatial trajectories, demonstrating that all vehicles navigate from the origin to their respective target NE points (marked with $\ast$). Finally, Fig. \ref{tutu4} explicitly validates that all vehicles' velocities converge to zero as their positions reach the NE. These observations verify the convergence results established in Theorem \ref{dingliyuan}.
%
%
\subsection{Experimental Results of Algorithm 2}\label{suan2li}
In practical situations, the communication topology among vehicles may change over time, owing to the unexpected phenomena, such as link failure, signal changes and repairs of components.\begin{figure}[!htpb]
\centerline{\includegraphics[width=0.8\linewidth]{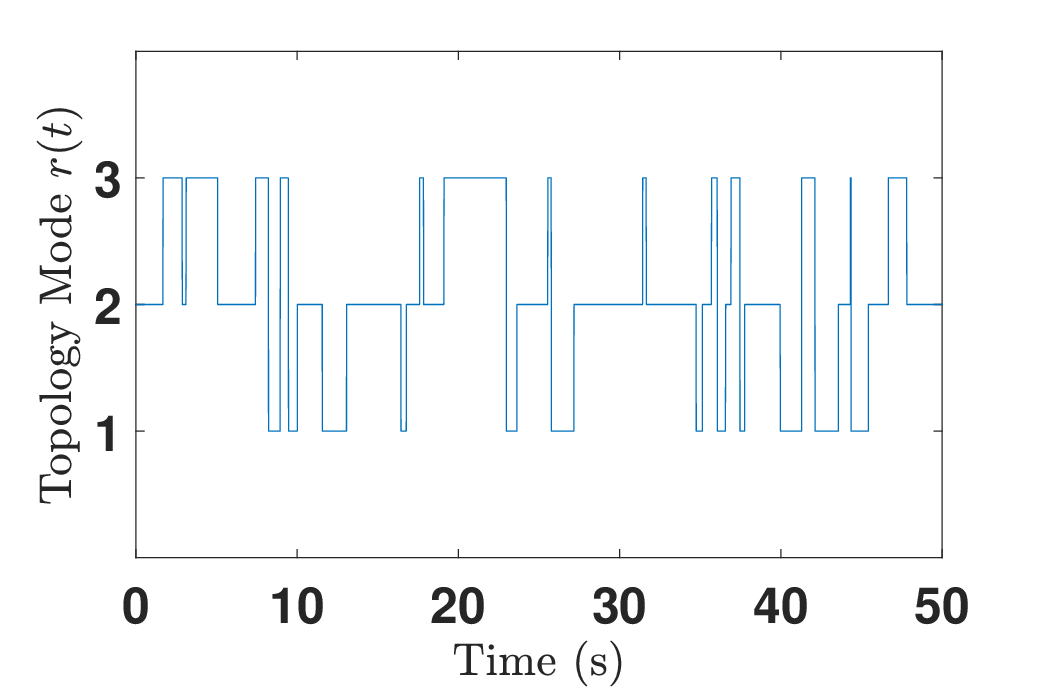}}
\caption{Switching signal with three modes.}\label{tutu6}
\vspace{-0.5cm}
\end{figure}\begin{figure}[!h]
\centering{
\begin{minipage}{0.8\linewidth}
\includegraphics[width=\linewidth]{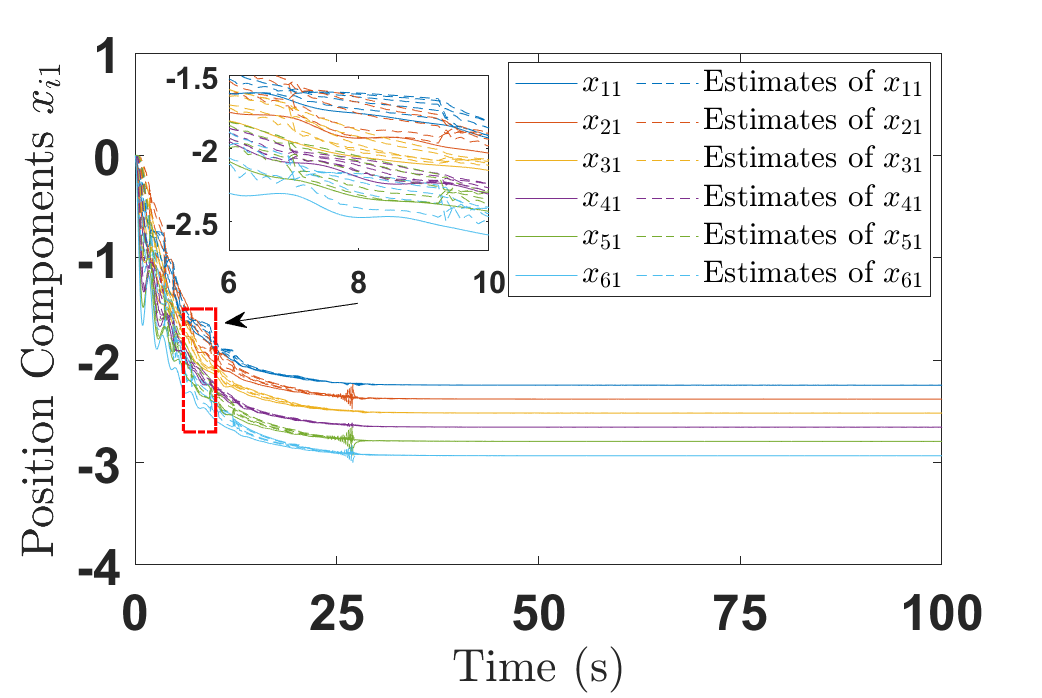}
\end{minipage}}
\centering{
\begin{minipage}{0.8\linewidth}
\includegraphics[width=\linewidth]{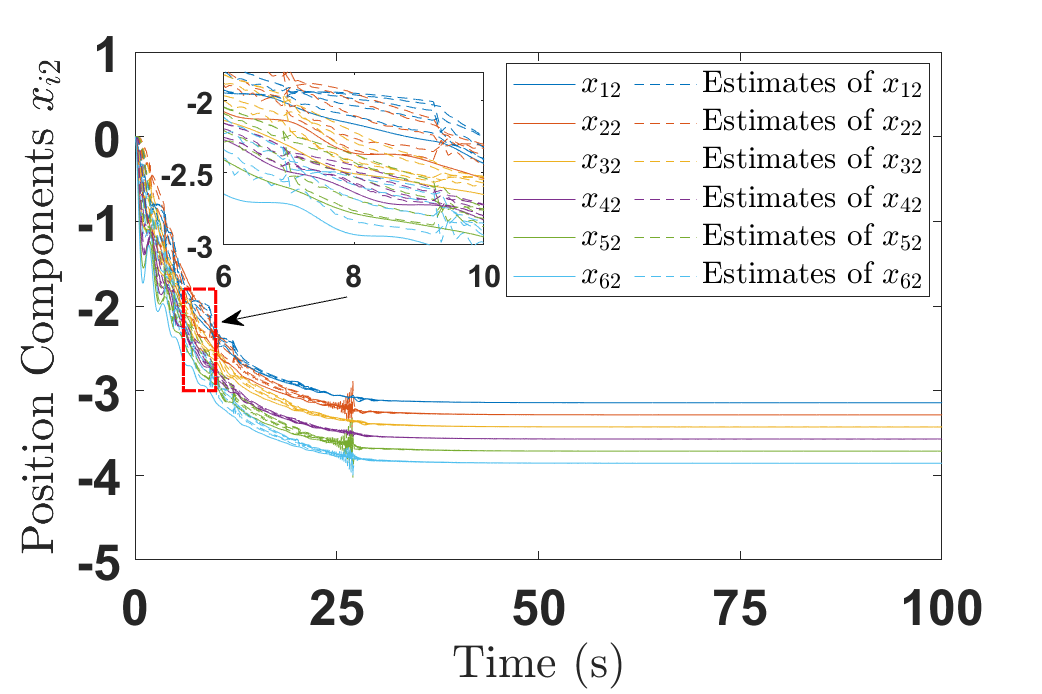}
\end{minipage}}
	\caption{$x_{i1}(t)$, $x_{i2}(t)$ for the position trajectories produced by Algorithm 2 with event-triggered condition \eqref{chufa}.}\label{tutu7}
\end{figure} In order to describe the time-varying phenomena, in this subsection, semi-Markov switching topologies are considered. In other words, the communication network among vehicles is represented by a set of undirected switching topologies (as illustrated in Fig. \ref{tutu5}), which are governed by\begin{figure}[!htpb]
\centerline{\includegraphics[width=0.8\linewidth]{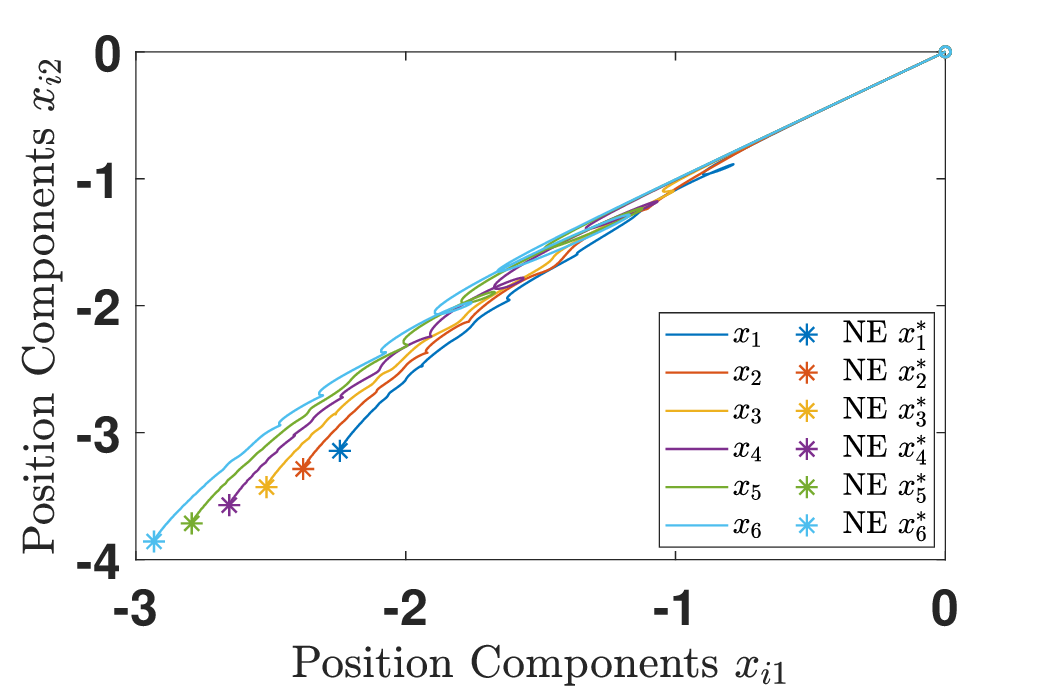}}
\caption{All the vehicles' trajectories in 2D plane.}\label{tutu8}
\label{f12}
\vspace{-0.3cm}
\end{figure}\begin{figure}[!htpb]
\centerline{\includegraphics[width=0.8\linewidth]{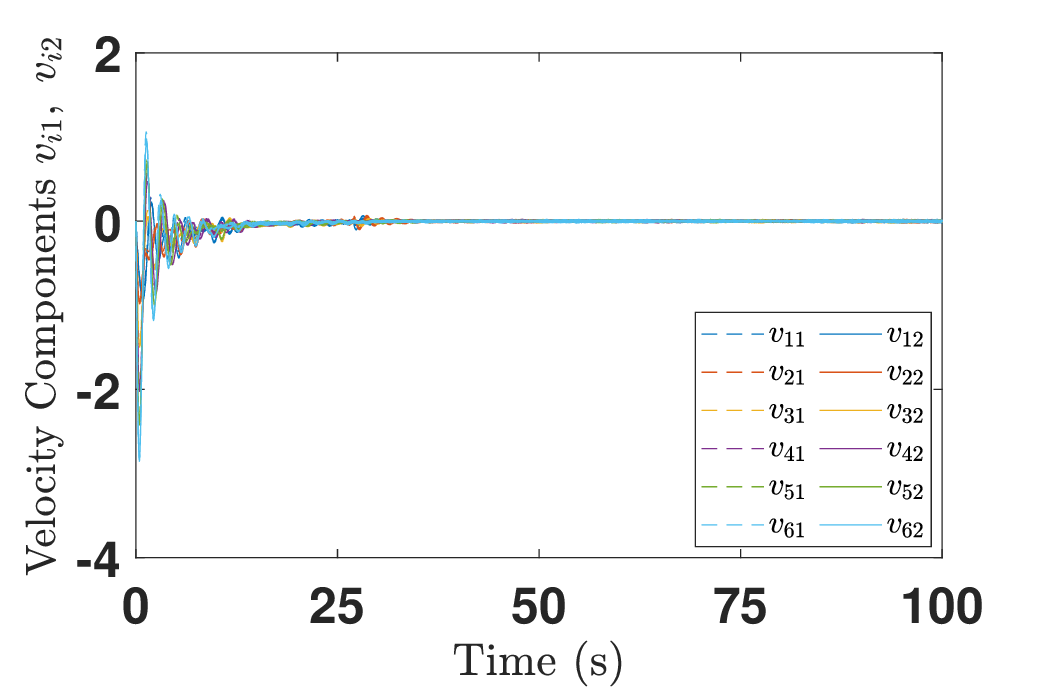}}
\caption{$v_{i1}(t)$, $v_{i2}(t)$ for the velocity trajectories produced by Algorithm 2 with event-triggered condition \eqref{chufa}.}\label{tutu9}
\label{f12}
\vspace{-0.5cm}
\end{figure}\begin{figure}[!htpb]
\centerline{\includegraphics[width=0.8\linewidth]{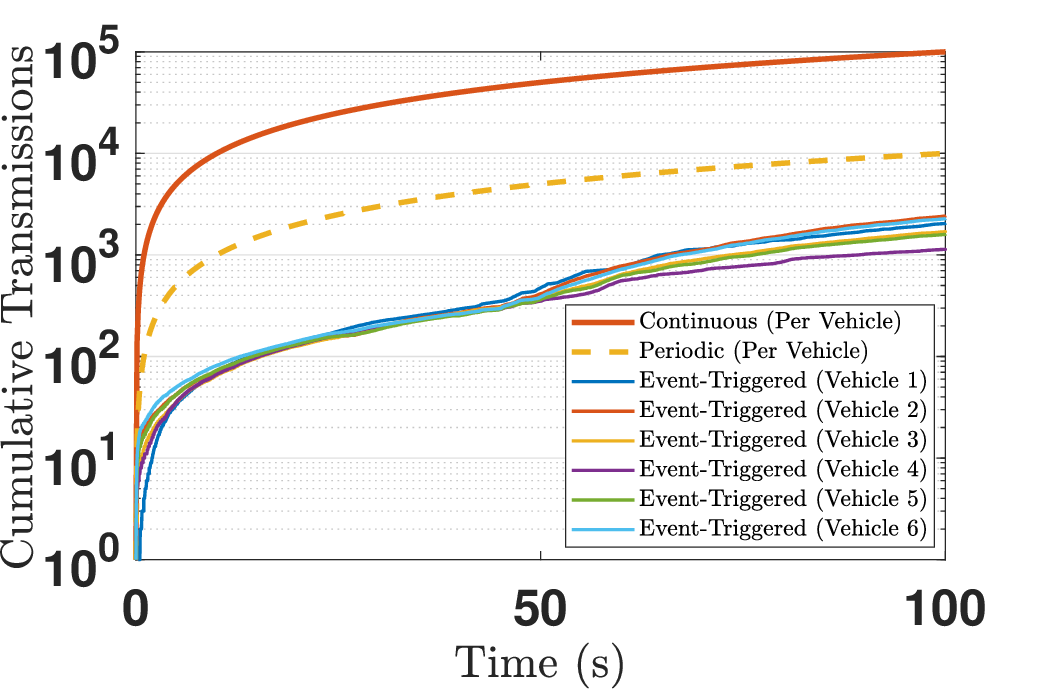}}
\caption{The cumulative number of information transmissions by three kinds of communication schemes on a semi-logarithmic scale.}\label{tutu10gai}
\end{figure}\begin{figure*}[!h]
\centering{
\begin{minipage}{0.3\linewidth}
\includegraphics[width=0.9\linewidth]{test1.png}

\end{minipage}}
\centering{
\begin{minipage}{0.3\linewidth}
\includegraphics[width=0.9\linewidth]{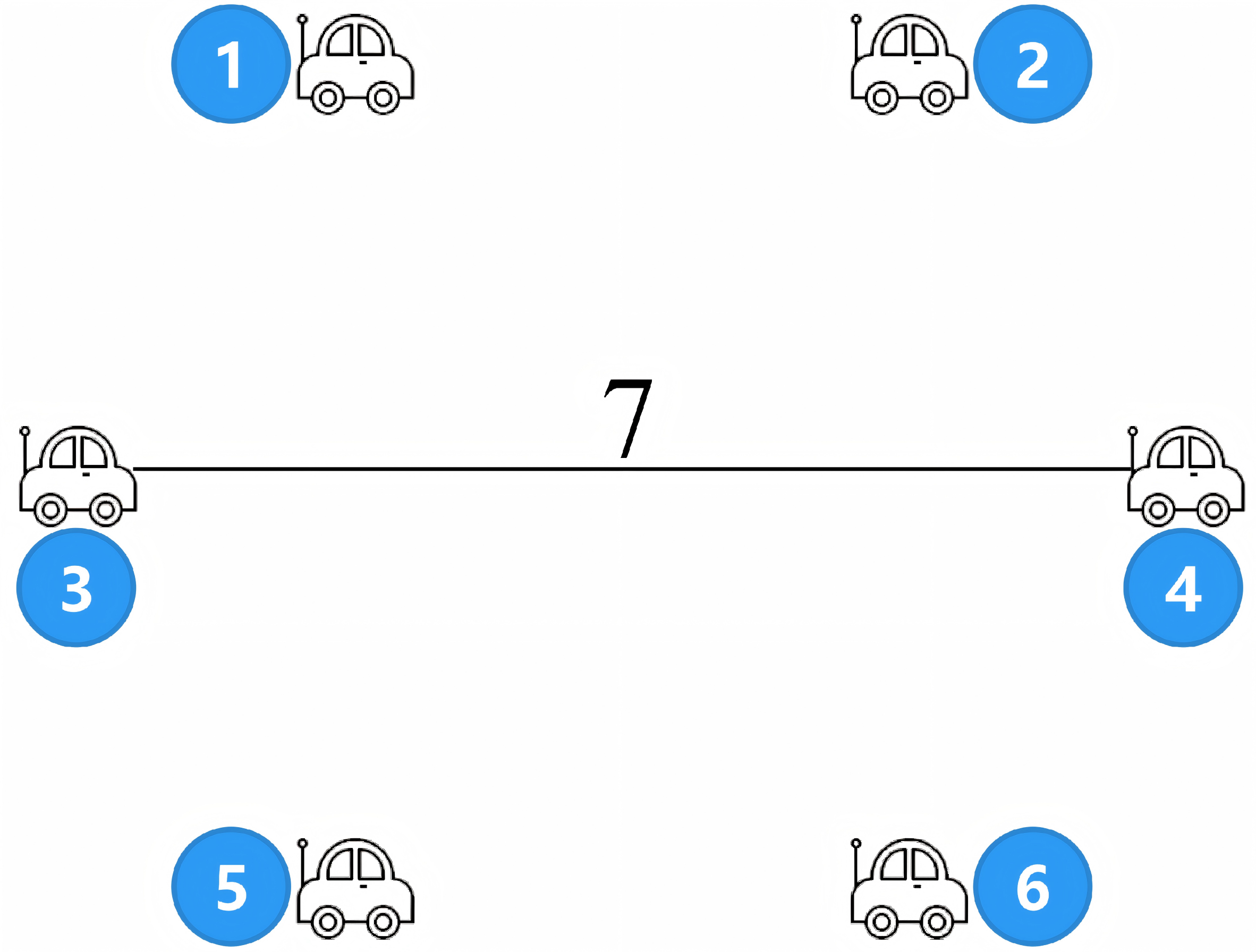}

\end{minipage}}
\centering{
\begin{minipage}{0.3\linewidth}
\includegraphics[width=0.9\linewidth]{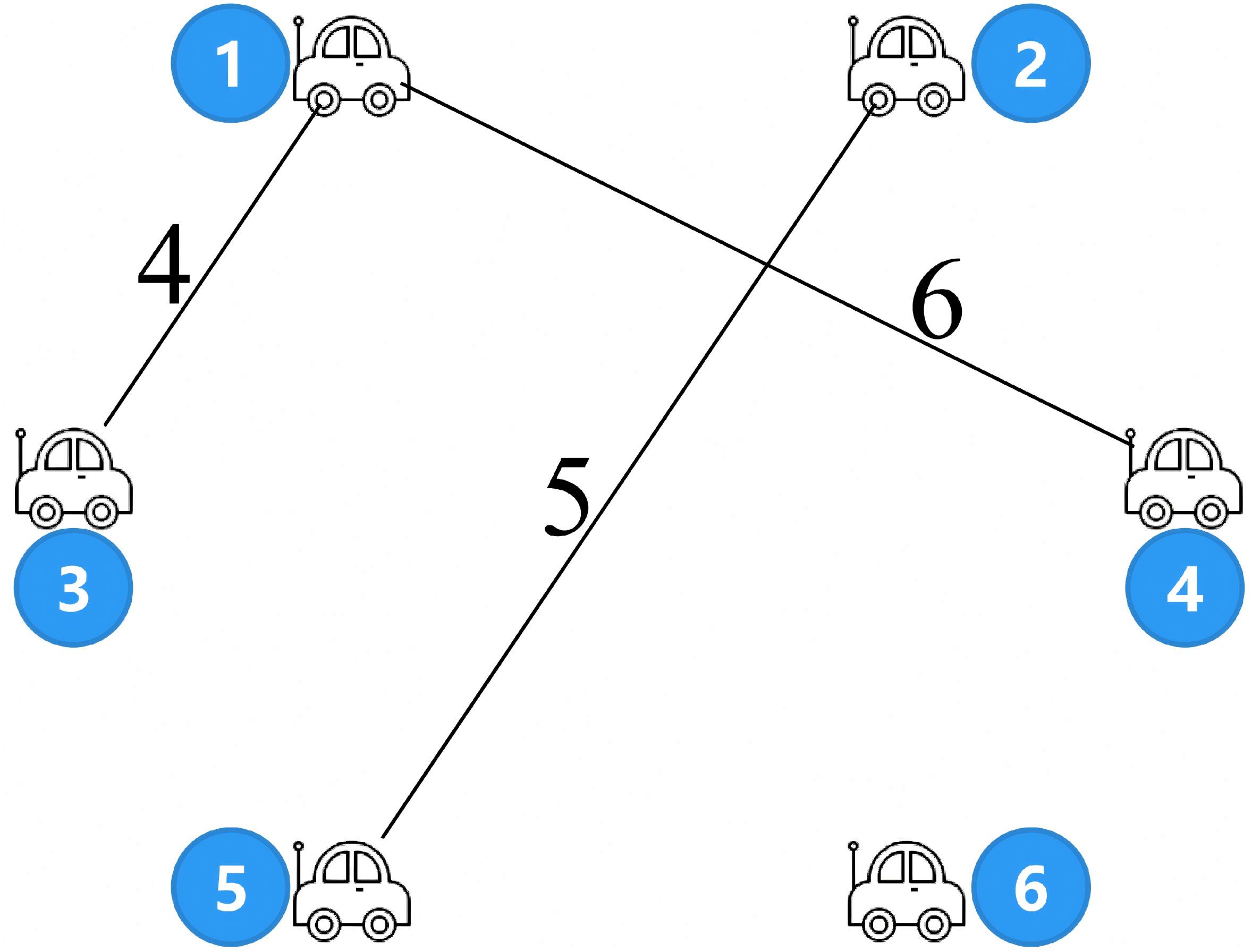}
\end{minipage}}
	\caption{Three possible switching topologies.}\label{tutu5gai}
\end{figure*} semi-Markov process with three different modes. Different from Markov process, for semi-Markov process, a general distribution, Weibull distribution (with probability density function $f(x ; \lambda, k)= \frac{k}{\lambda}\left(\frac{x}{\lambda}\right)^{k-1} e^{-(x / \lambda)^k}$ for $x\geq0$), is considered for sojourn time. The transition rates are time-varying, which are given as follows
\begin{equation*}
\begin{aligned}
&\iota_{11}(\vartheta) \in(-5.2,-1.3), \iota_{12}(\vartheta) \in(0.5,3), \iota_{13}(\vartheta) \in(0.8,2.2),\\
&\iota_{21}(\vartheta) \in(0.6,2), \iota_{22}(\vartheta) \in(-7,-2), \iota_{23}(\vartheta) \in(1.4,5),\\
&\iota_{31}(\vartheta) \in(0.3,5.6), \iota_{32}(\vartheta) \in(0.7,4.4), \iota_{33}(\vartheta) \in(-10,-1).
\end{aligned}
\end{equation*}

Although leveraging network information and interaction techniques improves both accuracy and efficiency, it can also increase the consumption of energy for constant communication. Thus, event-triggered mechanism is also introduced into the connectivity control game. Let sampling period $h=0.01s$ and the parameters $\zeta_1=\zeta_3=\zeta_5=0.1,~\zeta_2=\zeta_4=\zeta_6=0.05$. According to the linear matrix inequalities in Theorem \ref{dinglili3} and its equivalent transformation shown in \cite[Corollary 1]{dai2018event}, the event-triggered parameter matrix $\Phi$ and the feedback matrices $K(m),~ m\in\mathbb{S}$ are obtained, respectively. The simulation results produced by Algorithm 2 are shown in Figs. \ref{tutu6}-\ref{tutu10gai}.

Setting the scale parameter $\lambda=1$ and the shape parameter $k=1.5$ in Weibull distribution, we can get a semi-Markov switching signal shown in Fig. \ref{tutu6}. Fig. \ref{tutu7} displays the position trajectories of all the vehicles, indicating that their positions converge to the NE of the noncooperative game \eqref{li1}, while the position estimators accurately track the actual positions. Compared with the continuous communication scenario in Fig. \ref{tutu2}, the trajectories under the semi-Markov switching topologies and event-triggered mechanism exhibit more severe fluctuations. Nevertheless, these disturbance-induced oscillations are entirely eliminated around $t \approx 40$s, which further corroborates the disturbance rejection established in Lemma \ref{yinlili1}.
%
%
To provide a clearer spatial perspective, Fig. \ref{tutu8} demonstrates that despite the sparse communication, all the vehicles successfully navigate to their target positions in the 2D plane. Furthermore, the velocity trajectories in Fig. \ref{tutu9} confirm that all vehicles' velocities converge to zero despite facing severe transient oscillations. This substantiates the convergence results derived in Theorem \ref{dinglili3}. Moreover, to validate the communication reduction capability of the sampled-data-based event-triggered mechanism compared to continuous and periodic communication schemes, the cumulative number of information transmissions is plotted in Fig. \ref{tutu10gai} on a semi-logarithmic scale. It clearly shows that the proposed mechanism only demands communication updates when activated by the state-dependent threshold, effectively preventing the explosive accumulation of transmission counts seen in time-triggered schemes. Consequently, it demonstrates that the amount of information transmission is significantly reduced by approximately 98\% compared to the continuous communication and by approximately 80\% compared to the periodic communication, thereby maximizing the utilization efficiency of limited network resources.

In summary, the simulation results validate that Algorithm 2 achieves finite-time interference rejection and NE seeking for second-order players. Furthermore, the integration of the semi-Markov switching topologies and the sampled-data-based event-triggered mechanism endows Algorithm 2 with exceptional resilience against link failures and applicability in resource-constrained environments.
%
\subsection{Comparison with State-of-the-Art}
To facilitate a comparison, the connectivity control game \eqref{li1} is also addressed utilizing the algorithm proposed in \cite{fang2019distributed}. In this comparative scenario, external disturbances are excluded to isolate and independently evaluate the algorithm's effectiveness against semi-Markov switching topologies. As discussed in Remark \ref{zhujie33}, realistic communication networks often suffer from intermittent disconnections due to link failures. Therefore, the switching topologies are configured as illustrated in Fig. \ref{tutu5gai}, where the first graph is connected while the remaining two are disconnected. Under the same parameter settings as in Subsection \ref{suan2li}, but adjusting the shape parameter of the Weibull distribution to $k=0.5$ to induce a severe heavy-tailed semi-Markov process, the position trajectories produced by Algorithm 2 are presented in Fig. \ref{tutu7gai}. The results demonstrate that Algorithm 2 drives all vehicles' positions to the NE despite the prolonged communication failures.
\begin{figure}[!h]
\centering{
\begin{minipage}{0.8\linewidth}
\includegraphics[width=\linewidth]{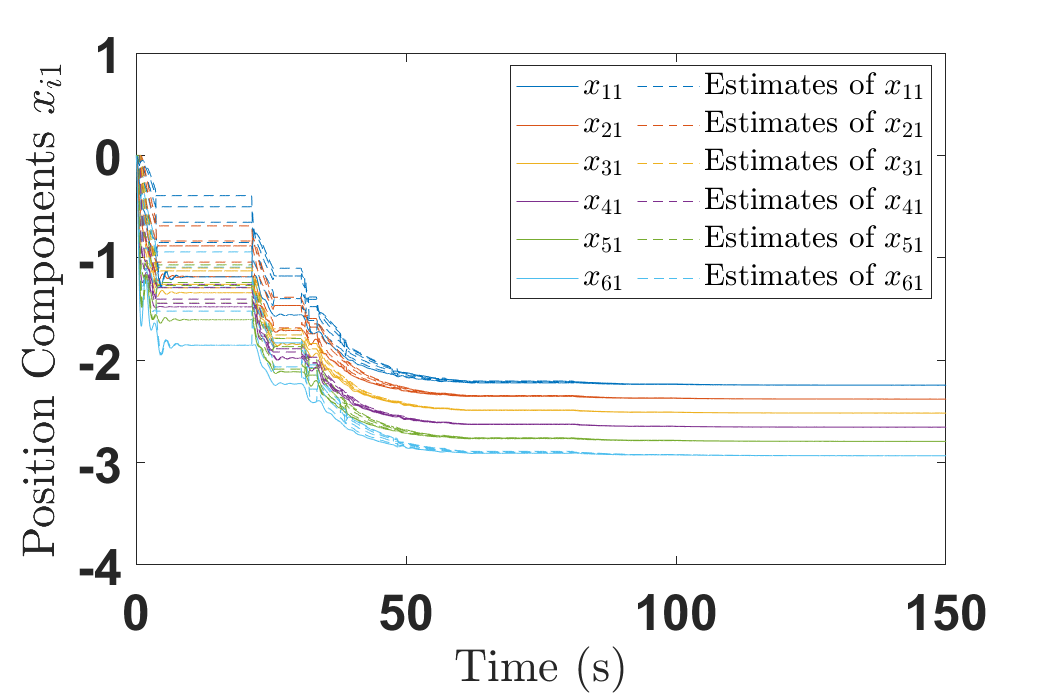}
\end{minipage}}
\centering{
\begin{minipage}{0.8\linewidth}
\includegraphics[width=\linewidth]{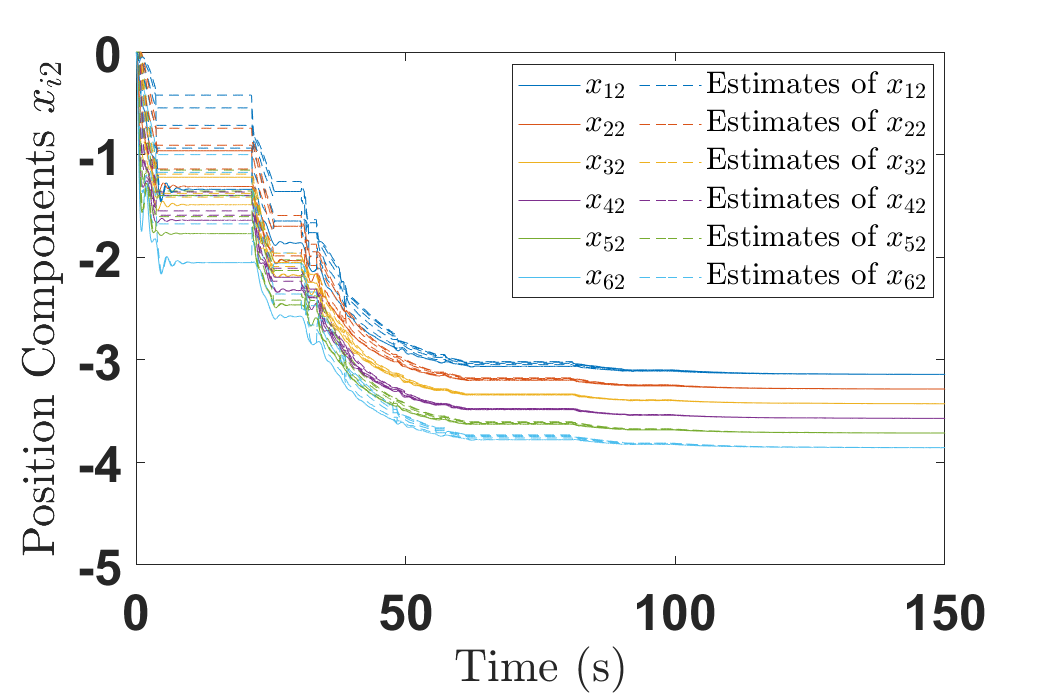}
\end{minipage}}
	\caption{$x_{i1}(t)$, $x_{i2}(t)$ for the position trajectories produced by Algorithm 2 with event-triggered condition \eqref{chufa}.}\label{tutu7gai}
\end{figure}

In stark contrast, under the identical semi-Markov switching sequence and setting the parameters ($\delta=0.005$, $\alpha_i=1, i\in\mathfrak{N}$, $T=0.001$, $\gamma=20$), the trajectories produced by the algorithm in \cite{fang2019distributed} are presented in Fig. \ref{tutu7semi}. It can be seen that even when the simulation is extended to $T_{\max}=500$s, the estimation variables fail to achieve consensus. This performance degradation fundamentally stems from the heavy-tailed property of the semi-Markov chain, which occasionally traps the communication graph in disconnected modes for extended durations. \begin{figure}[!h]
\centering{
\begin{minipage}{0.8\linewidth}
\includegraphics[width=\linewidth]{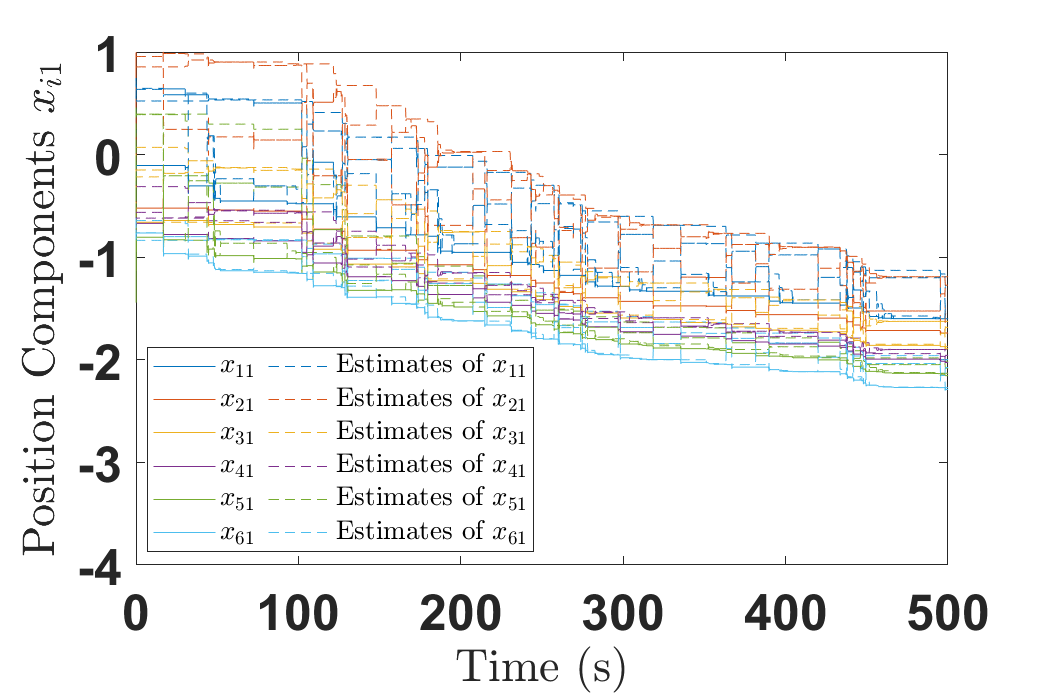}
\end{minipage}}
\centering{
\begin{minipage}{0.8\linewidth}
\includegraphics[width=\linewidth]{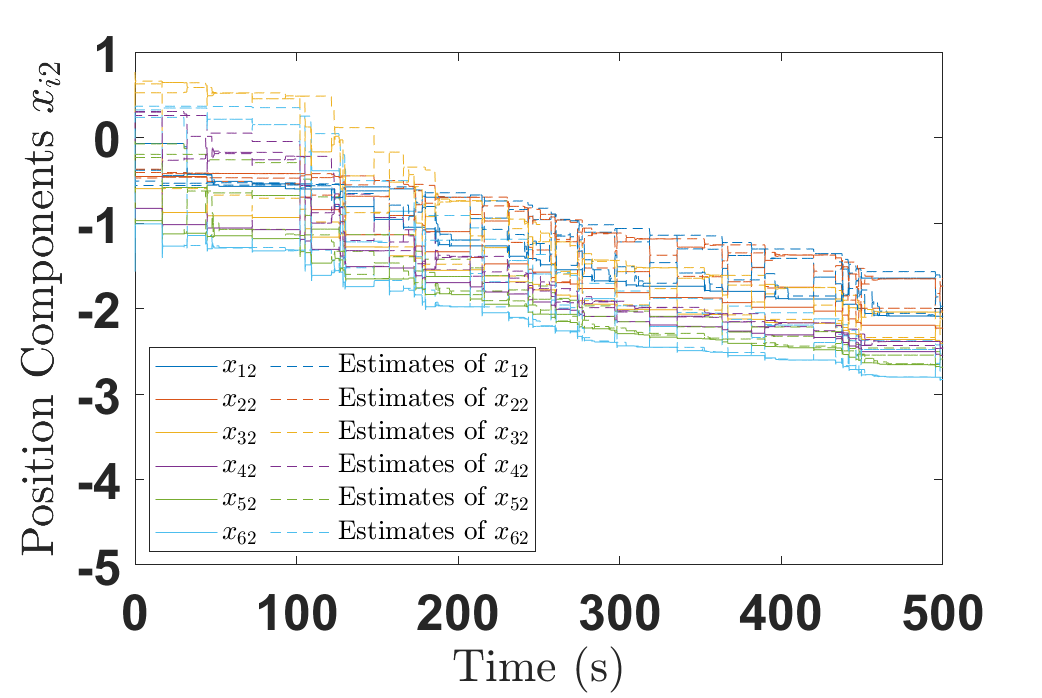}
\end{minipage}}
	\caption{$x_{i1}(t)$, $x_{i2}(t)$ for the position trajectories produced by the algorithm in \cite{fang2019distributed} under semi-Markov switching topologies.}\label{tutu7semi}
\vspace{-0.5cm}
\end{figure}

To further corroborate this analysis, we degenerate the semi-Markov chain into a standard, memoryless Markov chain. As depicted in Fig. \ref{tutu7mar}, without the heavy-tailed sojourn times, the algorithm in \cite{fang2019distributed} regains its tracking ability and reaches the NE within a short period.
\begin{figure}[!h]
\centering{
\begin{minipage}{0.8\linewidth}
\includegraphics[width=\linewidth]{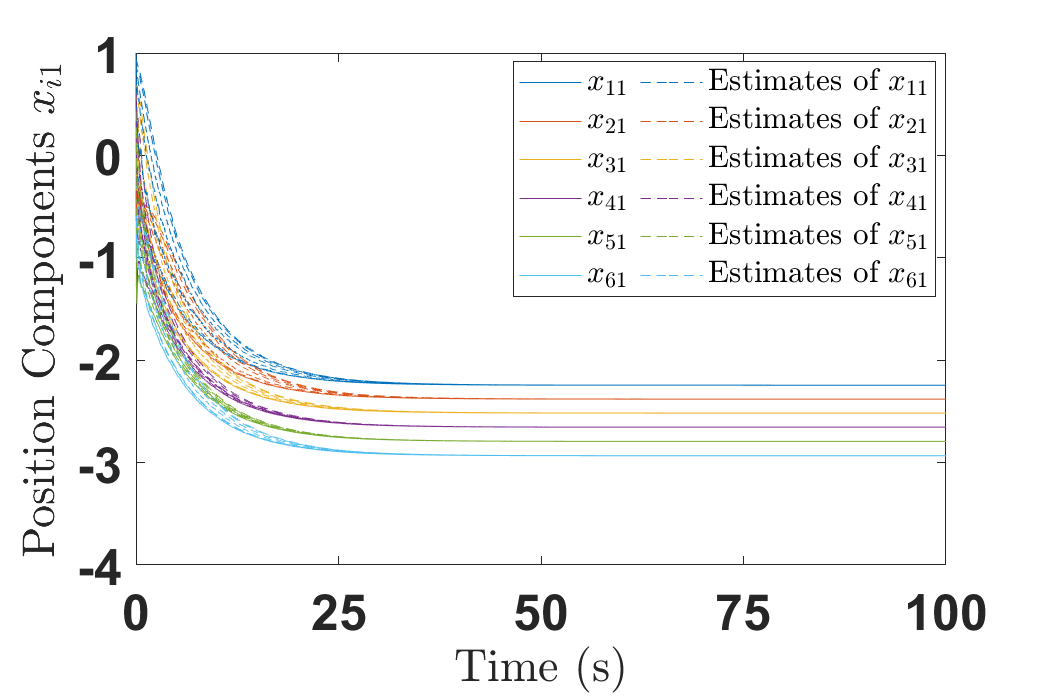}
\end{minipage}}
\centering{
\begin{minipage}{0.8\linewidth}
\includegraphics[width=\linewidth]{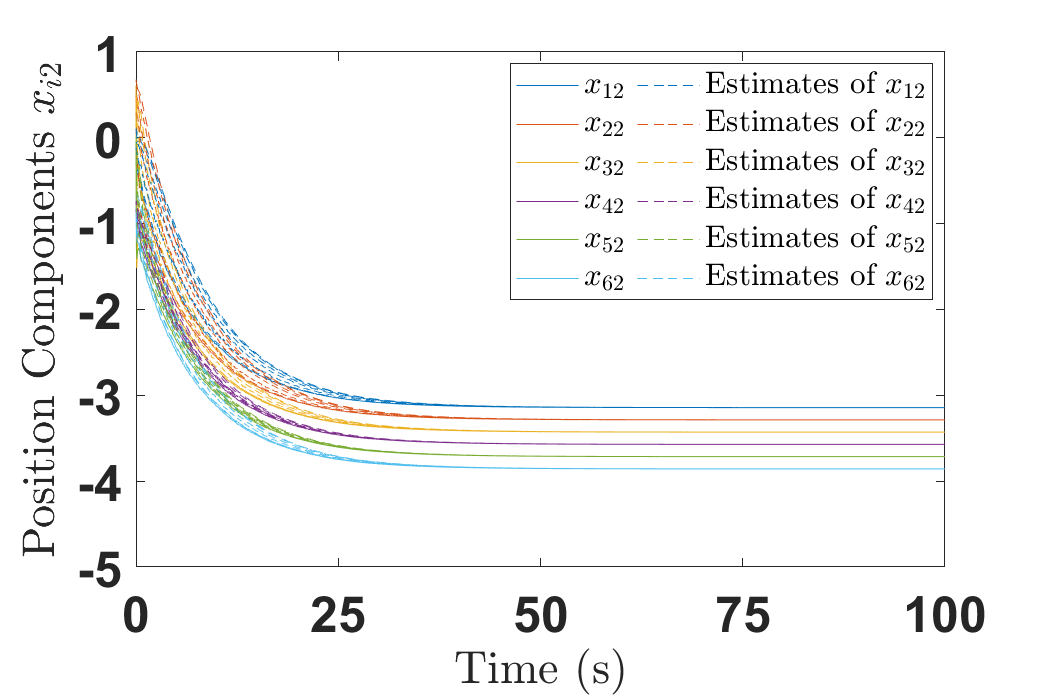}
\end{minipage}}
	\caption{$x_{i1}(t)$, $x_{i2}(t)$ for the position trajectories produced by the algorithm in \cite{fang2019distributed} under Markov switching topologies}\label{tutu7mar}
\vspace{-0.5cm}
\end{figure}

These comparative results comprehensively demonstrate the effectiveness of Algorithm 2 in terms of maintaining stability and NE seeking under intermittently disconnected networks governed by heavy-tailed semi-Markov processes.

%
	\section{Conclusion}\label{sec6}
This paper considered robust NE seeking subject to external disturbances and uncertain dynamics. By combining a supertwisting-based ISMC scheme with average consensus tracking, these interferences are rejected in finite time. Furthermore, to overcome link failures and high communication overhead, semi-Markov switching topologies and a sampled-data-based event-triggered mechanism were introduced. A novel leader-follower protocol is designed to guarantee the mean-square convergence of players' action estimations. Finally, application to a connectivity control game verified the algorithms' robustness and convergence.

\appendices
\section{Proof of Theorem \ref{dinglili3}}\label{fulu1}
For convenience, in the following, $\frac{\text{d}{\delta}}{\text{d}\tau}$ is still denoted by $\dot{\delta}$. Define the candidate Lyapunov function as
\begin{small}\begin{equation*}
\begin{aligned}
V(\tau, \delta _{\tau}, \dot{\delta}_ {\tau}, r(\tau))=&V_1(\tau, \delta _{\tau}, \dot{\delta}_ {\tau}, r(\tau))+V_2(\tau, \delta _{\tau}, \dot{\delta}_ {\tau}) \\
&+V_3(\tau, \delta _{\tau}, \dot{\delta}_ {\tau}), ~\forall \tau \in\left[\frac{q h}{\varepsilon},\frac{(q+1) h}{\varepsilon}\right),
\end{aligned}
\end{equation*}\end{small}in which
\begin{equation*}
\begin{aligned}
V_1=&\delta^{\rm T}(\tau)\left(\mathbf{I}_{Nn} \otimes \mathds{P}(r(\tau))\right) \delta(\tau),\\
V_2=&\int_{\tau-\frac{h}{\varepsilon}}^\tau \delta^{\rm T}(s)\left(\mathbf{I}_{Nn} \otimes \mathds{Q}\right) \delta(s) \text{d} s\\
&+\int_{\tau-\varsigma(\tau)}^\tau \delta^{\rm T}(s)\left(\mathbf{I}_{Nn} \otimes \mathds{U}\right) \delta(s) \text{d} s,\\
V_3=&\frac{h}{\varepsilon} \int_{-\frac{h}{\varepsilon}}^0 \int_{\tau+\varpi}^\tau \dot{\delta}^{\rm T}(s)\left(\mathbf{I}_{Nn}\otimes \mathds{R}\right) \dot{\delta}(s) \text{d} s \text{d} \varpi.
\end{aligned}
\end{equation*}
Declare the weak infinitesimal operator $\mathfrak{F}$ of $V$ as follows
\begin{equation*}
\begin{aligned}
&\mathfrak{F}V(\tau, \delta _{\tau}, \dot{\delta}_ {\tau}, r(\tau))=\lim _{\Delta \rightarrow 0^{+}} \frac{1}{\Delta}\bigg\{\mathbb{E}\{V(\tau+\Delta, \delta_{ \tau+\Delta}, \dot{\delta}_{ \tau+\Delta}, \\
&r(\tau+\Delta)) \mid\delta _\tau, r(\tau)\}-V(\tau, \delta _\tau, \dot{\delta}_\tau, r(\tau))\bigg\}.
\end{aligned}
\end{equation*}
Following directly from the same expression in the proof of Theorem 1 in \cite{miao2014mean}, we can get
\begin{small}\begin{equation*}
\begin{aligned}
\mathfrak{F} V_1=&\sum_{a=1}^s \lambda_{m n}(\vartheta) \delta^{\rm T}(\tau)\left(\mathbf{I}_{Nn} \otimes \mathds{P}(a)\right) \delta(\tau)\\
&+\dot{\delta}^{\rm T}(\tau)\left(\mathbf{I}_{Nn} \otimes \mathds{P}(m)\right) \delta(\tau)+\delta^{\rm T}(\tau)\left(\mathbf{I}_{Nn}\otimes \mathds{P}(m)\right)\dot{\delta}(\tau).
\end{aligned}
\end{equation*}\end{small}Then, for $V_2(\tau, \delta _{\tau}, \dot{\delta}_ {\tau})$ and $V_3(\tau, \delta _{\tau}, \dot{\delta}_ {\tau})$, one obtains
\begin{equation*}
\begin{aligned}
\mathfrak{F} V_2=&\delta^{\rm T}(\tau)\left[\mathbf{I}_{Nn} \otimes(\mathds{Q+U})\right] \delta(\tau)\\
&-\delta^{\rm T}(\tau-\frac{h}{\varepsilon}) (\mathbf{I}_{Nn}\otimes \mathds{Q})\delta(\tau-\frac{h}{\varepsilon}),\\
\mathfrak{F}V_3=&\frac{h^2}{\varepsilon^2} \dot{\delta}^{\rm T}(\tau)\left(\mathbf{I}_{Nn}\otimes \mathds{R}\right) \dot{\delta}(\tau)\\
&-\frac{h}{\varepsilon} \int_{\tau-\frac{h}{\varepsilon}}^\tau \dot{\delta}^{\rm T}(s)\left(\mathbf{I}_{Nn} \otimes R\right) \dot{\delta}(s) \text{d} s.
\end{aligned}
\end{equation*}
Based on Jensen's inequality, one derives
\begin{equation*}
\begin{aligned}
&-\frac{h}{\varepsilon} \int_{\tau-\frac{h}{\varepsilon}}^\tau \dot{\delta}^{\rm T}(s)\left(\mathbf{I}_{Nn} \otimes \mathds{R}\right) \dot{\delta}(s) \text{d} s\\
=&-\frac{h}{\varepsilon} \int_{\tau-\frac{h}{\varepsilon}}^{\tau-\varsigma(\tau)} \dot{\delta}^{\rm T}(s)\left(\mathbf{I}_{Nn} \otimes \mathds{R}\right) \dot{\delta}(s) \text{d} s\\
&-\frac{h}{\varepsilon} \int_{\tau-\varsigma(\tau)}^\tau \dot{\delta}^{\rm T}(s)\left(\mathbf{I}_{Nn}\otimes \mathds{R}\right) \dot{\delta}(s) \text{d} s\\
\leq&-\frac{h}{h-\varepsilon\varsigma(\tau)} \xi^{\rm T}(\tau) \text{\uppercase\expandafter{\romannumeral2}}_{23}^{\rm T}\left(\mathbf{I}_{Nn}\otimes \mathds{R}\right) \text{\uppercase\expandafter{\romannumeral2}}_{23} \xi(\tau)\\
&-\frac{h}{\varepsilon\varsigma(\tau)} \xi^{\rm T}(\tau)\text{\uppercase\expandafter{\romannumeral2}}_{12}^{\rm T}(\mathbf{I}_{Nn}\otimes \mathds{R})\text{\uppercase\expandafter{\romannumeral2}}_{12}\xi(\tau),
\end{aligned}
\end{equation*}
where $\xi(\tau)=\operatorname{col}(\delta(\tau), \delta(\tau-\varsigma(\tau)), \delta(\tau-\frac{h}{\varepsilon}), e(\tau-\varsigma(\tau)))$.\\
According to Lemma \ref{yinli32}, there is
\begin{equation*}
\begin{aligned}
&-\frac{h}{\varepsilon} \int_{\tau-\frac{h}{\varepsilon}}^\tau \dot{\delta}^{\rm T}(s)\left(\mathbf{I}_{Nn} \otimes \mathds{R}\right) \dot{\delta}(s) \text{d} s\\
\leq&-\xi^{\rm T}(\tau)\left[\begin{array}{l}
\text{\uppercase\expandafter{\romannumeral2}}_{23} \\
\text{\uppercase\expandafter{\romannumeral2}}_{12}
\end{array}\right]^{\rm T}\left[\begin{array}{cc}
\mathbf{I}_{Nn} \otimes \mathds{R} & \mathds{S} \\
\mathds{S} & \mathbf{I}_{Nn} \otimes \mathds{R}
\end{array}\right]\left[\begin{array}{l}
\text{\uppercase\expandafter{\romannumeral2}}_{23} \\
\text{\uppercase\expandafter{\romannumeral2}}_{12}
\end{array}\right] \xi(\tau).
\end{aligned}
\end{equation*}
On account of event-triggered condition \eqref{chufa}, it gives that
\begin{equation*}
\begin{aligned}
&e^{\rm T}(q h)\left(\mathbf{I}_{N}\otimes \Phi\right) e(q h)=\sum_{i=1}^N e_i^{\rm T}(q h) \Phi e_i(q h) \\
\leq& \sum_{i=1}^N \zeta_i z_i^{\rm T}(q h)\Phi z_i(q h)=z^{\rm T}(q h)(\Lambda \otimes \Phi)z(qh),
\end{aligned}
\end{equation*}
and
\begin{equation*}
z(q h) \left.=\left((L(m)\otimes \mathbf{I}_N+A_0(m))\otimes\mathbf{I}_n\right)[\delta(q h)+e (qh)\right],
\end{equation*}
where $z:=\col(z_1,z_2,\ldots,z_N)$. Thus, we can conclude that
\begin{equation*}
\begin{aligned}
&e^{\rm T}(\tau-\varsigma(\tau))\left(\mathbf{I}_{N} \otimes \Phi\right) e(\tau-\varsigma(\tau))\\
\leq&(\delta(\tau-\varsigma(\tau))+e(\tau-\varsigma(\tau)))^{\rm T} \left(\mathcal{H}(m)^{\rm T} (\Lambda \otimes \Phi) \mathcal{H}(m)\right)\\
&\times(\delta(\tau-\varsigma(\tau))+e(\tau-\varsigma(\tau))).
\end{aligned}
\end{equation*}
Not only that, based on the definition of $\delta$, we can get that $\dot{\delta}(\tau)=\mathds{B}(m)\xi(\tau)$. Thus, combining the above discussion, we derive that
\begin{small}\begin{equation*}
\begin{aligned}
\mathfrak{F}V \leq& \sum_{a=1}^s \lambda_{m n}(\vartheta) \delta^{\rm T}(\tau)\left(\mathbf{I}_{Nn}\otimes \mathds{P}(a)\right) \delta(\tau)\\
&+\dot{\delta}^{\rm T}(\tau)\left(\mathbf{I}_{Nn}\otimes \mathds{P}(m)\right) \delta(\tau)+\delta^{\rm T}(\tau)\left(\mathbf{I}_{Nn} \otimes \mathds{P}(m)\right) \dot{\delta}(\tau)\\
&+\delta^{\rm T}(\tau)\left[\mathbf{I}_{Nn} \otimes(\mathds{Q+U})\right] \delta(\tau)+\frac{h^2}{\varepsilon^2} \dot{\delta}^{\rm T}(\tau)\left(\mathbf{I}_{Nn}\otimes \mathds{R}\right) \dot{\delta}(\tau)\\
&-\delta^{\rm T}(\tau-\frac{h}{\varepsilon}) (\mathbf{I}_{Nn}\otimes \mathds{Q})\delta(\tau-\frac{h}{\varepsilon})\\
&-\xi^{\rm T}(\tau)\left[\begin{array}{l}
\text{\uppercase\expandafter{\romannumeral2}}_{23} \\
\text{\uppercase\expandafter{\romannumeral2}}_{12}
\end{array}\right]^{\rm T}\left[\begin{array}{cc}
\mathbf{I}_{Nn} \otimes \mathds{R} & \mathds{S} \\
\mathds{S} & \mathbf{I}_{Nn} \otimes \mathds{R}
\end{array}\right]\left[\begin{array}{l}
\text{\uppercase\expandafter{\romannumeral2}}_{23} \\
\text{\uppercase\expandafter{\romannumeral2}}_{12}
\end{array}\right] \xi(\tau)\\
&-e^{\rm T}(\tau-\varsigma(\tau))\left(\mathbf{I}_{N} \otimes \Phi\right) e(\tau-\varsigma(\tau))\\
&+(\delta(\tau-\varsigma(\tau))+e(\tau-\varsigma(\tau)))^{\rm T} \left(\mathcal{H}(m)^{\rm T} (\Lambda \otimes \Phi) \mathcal{H}(m)\right)\\
&\times(\delta(\tau-\varsigma(\tau))+e(\tau-\varsigma(\tau)))\\
=&\xi^{\rm T}(\tau)\Psi(m) \xi(\tau),
\end{aligned}
\end{equation*}\end{small}where $\Psi(m)=\Xi_1(m)+\frac{h^2}{\varepsilon^2}\mathds{B}(m)^{\rm T}\left(\mathbf{I}_{Nn} \otimes \mathds{R}\right)\mathds{B}(m)$.

Based on the inequalities \eqref{juzhen1} and \eqref{juzhen2}, one derives
\begin{equation*}
\mathfrak{F}V(\tau, \delta _{\tau}, \dot{\delta}_ {\tau}, r(\tau))<0,~\tau\in\left[\frac{qh}{\varepsilon},\frac{(q+1)h}{\varepsilon}\right).
\end{equation*}
Then, it also means that, for sufficiently small $\epsilon$, we have
\begin{equation*}
\mathfrak{F}V(\tau, \delta _{\tau}, \dot{\delta}_ {\tau}, r(\tau))<-\epsilon\delta^{\rm T}(\tau)\delta(\tau),~\tau\in\left[\frac{qh}{\varepsilon},\frac{(q+1)h}{\varepsilon}\right).
\end{equation*}
For convenience, denote $\tau_q=\frac{qh}{\varepsilon}$ and $\tau_{q+1}=\frac{(q+1)h}{\varepsilon}$. Then, according to Dynkin's formula, there is
\begin{small}\begin{equation*}
\begin{aligned}
&\mathbb{E}\{V(\tau_{q+1}^-  ,\delta_{\tau_{q+1}^-}, \dot{\delta}_{\tau_{q+1}^-} ,r(\tau_{q+1}^-))\}- \mathbb{E}\{V(\tau_{q}  ,\delta_{\tau_{q}}, \dot{\delta}_{\tau_{q}} ,r(\tau_{q}))\}\\
\leq&-\epsilon \mathbb{E}\left\{\int_{\tau_q}^{\tau_{q+1}^{-}}\|\delta(s)\|^2 \text{d} s\right\}.
\end{aligned}
\end{equation*}\end{small}
Similarly, one has
\begin{small}
\begin{equation*}
\begin{aligned}
&\mathbb{E}\{V(\tau_q^{-}, \delta_{\tau_q^{-}}, \dot{\delta}_{\tau_q^{-}}, r(\tau_q^{-}))\}-\mathbb{E}\{V(\tau_{q-1}, \delta_{\tau_{q-1}}, \dot{\delta}_{\tau_{q-1}}, r(\tau_{q-1}))\} \\
\leq&-\epsilon \mathbb{E}\left\{\int_{t_{q-1}}^{t_{q}^{-}}\|\delta(s)\|^2 \text{d} s\right\} \\
&~~~~~~~~~~~~~~~~~~~~~~~~~~~~~~~~~~\vdots \\
&\mathbb{E}\{V(\tau_1^{-}, \delta_{\tau_1^{-}}, \dot{\delta}_{\tau_1^{-}}, r(\tau_1^{-}))\}-\mathbb{E}\{V(0, \delta_0, \dot{\delta}_0, r(0))\} \\
\leq&-\epsilon \mathbb{E}\left\{\int_0^{t_1^{-}}\|\delta(s)\|^2 \text{d} s\right\}.
\end{aligned}
\end{equation*}
\end{small}
Due to the fact that $V_2(\tau, \delta _{\tau}, \dot{\delta}_ {\tau})\geq0$ and $\int_{\tau-\varsigma(\tau)}^\tau \delta^{\rm T}(s)\left(\mathbf{I}_{Nn}\otimes \mathds{U}\right) \delta(s) \text{d} s=0$ when $\tau=\frac{qh}{\varepsilon}$, we have $$\lim _{ \tau \rightarrow (\frac{q h}{\varepsilon})^{-}}V_2(\tau, \delta _{\tau}, \dot{\delta}_ {\tau}) \geq V_2(\tau, \delta _{\tau}, \dot{\delta}_ {\tau})|_{\tau=\frac{q h}{\varepsilon}}.$$ It also means that
\begin{equation*}
\mathbb{E}\{ V(\tau_q^-, \delta_{\tau_q^-},\dot{\delta}_{\tau_q^-}, r(\tau_q^-))\} \geq \mathbb{E}\{V(\tau_q, \delta_{\tau_q}, \dot{\delta}_{ \tau_q}, r(\tau_q))\}.
\end{equation*}
Therefore, based on the above inequalities, one can obtain
$$\sum_{q=0}^{\infty} \mathbb{E}\left\{\int_{\tau_ q}^{\tau_{q+1}^-}||\delta(s)||^2\text{d}s\right\}\leq\epsilon^{-1} \mathbb{E}\left\{ V\left(0, \delta_0, \dot{\delta}_0, r(0)\right) \right\}.$$
It is to say
$$\lim _{\gamma \rightarrow \infty} \mathbb{E}\left\{\int_0^\gamma\|\delta(s) \|^2 \text{d} s\right\}<\infty,$$
which also implies that $\lim _{\gamma \rightarrow \infty} \mathbb{E}\|\delta(\gamma)\|^2=0$. Thus, the state of the boundary-layer system \eqref{bianjie2} with event-triggered mechanism \eqref{chufa} achieves consensus in the mean-square sense.
\section{Proof of Corollary \ref{tuilun111}}\label{fulu2}
The only work left is to prove that the state of the system \eqref{bianjie2} achieves consensus in the sense of mean-square. Through a similar process in Theorem \ref{dinglili3}, we can get
\begin{equation*}
\mathfrak{F}V \leq\xi^{\rm T}(\tau)\Psi(m) \xi(\tau).
\end{equation*}
Based on the relaxed condition, $\mathfrak{F}V$ satisfies
\begin{equation*}
\mathfrak{F}V(\tau, \delta_{\tau}, \dot{\delta}_{\tau}, r(\tau)=m) \leq
\begin{cases}
-\imath_m \|\xi(\tau)\|^2, & m \in \mathbb{S}_c,\\
\jmath_m \|\xi(\tau)\|^2, & m \in \mathbb{S}_d.
\end{cases}
\end{equation*}
Define an indicator function $\chi_m(\tau)$ such that $\chi_m(\tau) = 1$ if $r(\tau) = m$, and $\chi_m(\tau) = 0$ otherwise. Then, the expectation of $\mathfrak{F}V$ along the system trajectory can be bounded by
\begin{equation}\label{ergodic_eq}
\mathbb{E}\left\{\mathfrak{F}V\right\} \leq \mathbb{E} \left\{ ( \sum_{m \in \mathbb{S}_c} (-\alpha_m \chi_m(\tau)) + \sum_{n \in \mathbb{S}_d} \beta_n \chi_n(\tau) ) \|\xi(\tau)\|^2 \right\}.
\end{equation}
Since the semi-Markov process $r(\tau)$ is ergodic, according to the ergodic theorem \cite{limnios2012semi}, the time average of the indicator function converges almost surely to the stationary probability distribution $\pi_m$, i.e.,
\begin{equation*}
\lim_{\Gamma \to \infty} \frac{1}{\Gamma} \int_0^{\Gamma} \chi_m(s) \text{d}s = \pi_m, \quad \text{a.s.}
\end{equation*}
Consequently, applying the Dynkin's formula yields
\begin{small}\begin{equation*}
\begin{aligned}
&\mathbb{E}\{V(\tau_{q+1}^-  ,\delta_{\tau_{q+1}^-}, \dot{\delta}_{\tau_{q+1}^-} ,r(\tau_{q+1}^-))\}- \mathbb{E}\{V(\tau_{q}  ,\delta_{\tau_{q}}, \dot{\delta}_{\tau_{q}} ,r(\tau_{q}))\}\\
\leq&\mathbb{E}\left\{\int_{\tau_q}^{\tau_{q+1}^{-}}( \sum_{m \in \mathbb{S}_c} (-\alpha_m \chi_m(s)) + \sum_{n \in \mathbb{S}_d} \beta_n \chi_n(s) ) \|\xi(s)\|^2 \text{d} s\right\}.
\end{aligned}
\end{equation*}\end{small}
Similarly, one has
\begin{small}
\begin{equation*}
\begin{aligned}
&\mathbb{E}\{V(\tau_q^{-}, \delta_{\tau_q^{-}}, \dot{\delta}_{\tau_q^{-}}, r(\tau_q^{-}))\}-\mathbb{E}\{V(\tau_{q-1}, \delta_{\tau_{q-1}}, \dot{\delta}_{\tau_{q-1}}, r(\tau_{q-1}))\} \\
\leq&\mathbb{E}\left\{\int_{t_{q-1}}^{t_{q}^{-}}( \sum_{m \in \mathbb{S}_c} (-\alpha_m \chi_m(s)) + \sum_{n \in \mathbb{S}_d} \beta_n \chi_n(s) ) \|\xi(s)\|^2 \text{d} s\right\} \\
&~~~~~~~~~~~~~~~~~~~~~~~~~~~~~~~~~~~~~~~~~~~~~~\vdots \\
&\mathbb{E}\{V(\tau_1^{-}, \delta_{\tau_1^{-}}, \dot{\delta}_{\tau_1^{-}}, r(\tau_1^{-}))\}-\mathbb{E}\{V(0, \delta_0, \dot{\delta}_0, r(0))\} \\
\leq&\mathbb{E}\left\{\int_0^{t_1^{-}}( \sum_{m \in \mathbb{S}_c} (-\alpha_m \chi_m(s)) + \sum_{n \in \mathbb{S}_d} \beta_n \chi_n(s) ) \|\xi(s)\|^2 \text{d} s\right\}.
\end{aligned}
\end{equation*}
\end{small}
Due to the fact that $V_2(\tau, \delta _{\tau}, \dot{\delta}_ {\tau})\geq0$ and $\int_{\tau-\varsigma(\tau)}^\tau \delta^{\rm T}(s)\left(\mathbf{I}_{Nn}\otimes \mathds{U}\right) \delta(s) \text{d} s=0$ when $\tau=\frac{qh}{\varepsilon}$, we have $$\lim _{ \tau \rightarrow (\frac{q h}{\varepsilon})^{-}}V_2(\tau, \delta _{\tau}, \dot{\delta}_ {\tau}) \geq V_2(\tau, \delta _{\tau}, \dot{\delta}_ {\tau})|_{\tau=\frac{q h}{\varepsilon}}.$$ It also means that
\begin{equation*}
\mathbb{E}\{ V(\tau_q^-, \delta_{\tau_q^-},\dot{\delta}_{\tau_q^-}, r(\tau_q^-))\} \geq \mathbb{E}\{V(\tau_q, \delta_{\tau_q}, \dot{\delta}_{ \tau_q}, r(\tau_q))\}.
\end{equation*}
Therefore, based on the above inequalities, one can obtain
\begin{equation*}
\begin{aligned}
&\sum_{q=0}^{\infty} \mathbb{E}\left\{\int_{\tau_ q}^{\tau_{q+1}^-}( \sum_{m \in \mathbb{S}_c} \alpha_m \chi_m(s) - \sum_{n \in \mathbb{S}_d} \beta_n \chi_n(s) ) \|\xi(s)\|^2\text{d}s\right\}\\
\leq&\mathbb{E}\left\{ V\left(0, \delta_0, \dot{\delta}_0, r(0)\right) \right\}.
\end{aligned}
\end{equation*}
Over an infinite horizon, the time fraction spent in each mode converges to its stationary probability. Thus, based on the following condition\begin{equation}\label{rho_cond}\rho := \sum_{m \in \mathbb{S}_c} \pi_m \alpha_m - \sum_{n \in \mathbb{S}_d} \pi_n \beta_n > 0,\end{equation}the overall expected drift of the system is strictly negative. This allows us to bound the infinite integral as$$\lim _{\gamma \rightarrow \infty} \mathbb{E}\left\{\int_0^\gamma\|\xi(s) \|^2 \text{d} s\right\} \leq \frac{1}{\rho}\mathbb{E}\left\{ V\left(0, \delta_0, \dot{\delta}_0, r(0)\right) \right\} <\infty,$$which also implies that $\lim _{\gamma \rightarrow \infty} \mathbb{E}\|\delta(\gamma)\|^2=0$. Thus, the state of the boundary-layer system \eqref{bianjie2} with event-triggered mechanism \eqref{chufa} achieves consensus in the mean-square sense. 
\section*{References}
\bibliography{markov}
\bibliographystyle{IEEEtran}
\end{document}